\documentclass[a4paper,12pt,reqno]{amsart}
\pdfoutput=1

\usepackage{amsmath,amssymb,amsthm,mathrsfs}
\usepackage{mathtools}
\usepackage[all,cmtip]{xy}
\usepackage{bm}
\usepackage[a4paper,margin=2cm]{geometry}
\usepackage{lmodern}
\usepackage[T1]{fontenc}
\usepackage{microtype}

\usepackage[shortlabels]{enumitem}

\newtheorem{theorem}{Theorem}[section]
\newtheorem*{theorem*}{Main Theorem}
\newtheorem{proposition}[theorem]{Proposition}
\newtheorem{lemma}[theorem]{Lemma}
\newtheorem{corollary}[theorem]{Corollary}
\theoremstyle{definition} 
\newtheorem{definition}[theorem]{Definition} 
\newtheorem{remark}[theorem]{Remark}

\makeatletter
\let\c@equation\c@theorem
\makeatother
\numberwithin{equation}{section}

\usepackage[hidelinks]{hyperref}

\newcommand{\lra}{\longrightarrow}
\newcommand{\Cat}{\mathbf{Cat}}

\newcommand{\Gray}{\operatorname{\mathbf{Gray}}}
\newcommand{\dbl}{\operatorname{\mathbf{Dbl}}}
\newcommand{\psdbl}{\operatorname{\mathbf{PsDbl}}}

\newcommand{\twocat}{\mathbf{2}\text{-}\mathbf{Cat}}

\newcommand{\Set}{\operatorname{\mathbf{Set}}}
\newcommand{\Hom}{\mathrm{\mathbf{Hom}}}
\newcommand{\Ps}{\mathrm{\mathbf{Ps}}}
\newcommand{\Bicat}{\mathbf{Bicat}}
\newcommand{\st}{\mathrm{\mathbf{st}}}

\newcommand{\cd}[2][]{\vcenter{\hbox{\xymatrix#1{#2}}}}

\newcommand{\dtwocell}[3][0.5]{\ar@{}[#2] \ar@{=>}?(#1)+/u  0.2cm/;?(#1)+/d 0.2cm/^{#3}}
\newcommand{\ltwocell}[3][0.5]{\ar@{}[#2] \ar@{=>}?(#1)+/r 0.2cm/;?(#1)+/l 0.2cm/_{#3}}

\makeatletter
\def\matrixobject@{%
  \edef \next@{={\DirectionfromtheDirection@ }}%
  \expandafter \toks@ \next@ \plainxy@
  \let\xy@@ix@=\xyq@@toksix@
  \xyFN@ \OBJECT@}
\let\xy@entry@@norm=\entry@@norm
\def\entry@@norm@patched{%
  \let\object@=\matrixobject@
  \xy@entry@@norm }
\AtBeginDocument{\let\entry@@norm\entry@@norm@patched}
\makeatother

\newcommand{\hdash}{\rotatebox[origin=c]{90}{$\vdash$}}

\newcommand{\vcong}{\rotatebox[origin=c]{-90}{$\cong$}}

\title{How strict is strictification?}
\author{Alexander Campbell}
\address{Centre of Australian Category Theory, Macquarie University, NSW 2109, Australia}
\email{alexander.campbell@mq.edu.au}

\date{5 October 2018}

\begin{document}

\begin{abstract}
The subject of this paper is the higher structure of the strictification adjunction, which relates the two fundamental bases of three-dimensional category theory:\ the $\Gray$-category of $2$-categories and the tricategory of bicategories. We show that -- far from requiring the full weakness provided by the definitions of tricategory theory -- this adjunction can be \emph{strictly} enriched over the symmetric closed multicategory of bicategories defined by Verity. Moreover, we show that this adjunction underlies an adjunction of bicategory-enriched symmetric multicategories.  
 An appendix introduces the symmetric closed multicategory of pseudo double categories, into which Verity's symmetric multicategory of bicategories embeds fully.
\end{abstract}

\maketitle

\tableofcontents

\section{Introduction} \label{intro}
The fundamental coherence theorem for bicategories (see \cite[\S 1.4]{MR1261589}) states that every\- bicategory is biequivalent to a $2$-category (that is, a category enriched over the cartesian closed category $\Cat$ of categories). Moreover, as is shown by the extensive literature on two-dimensional category theory, 
the \emph{category theory} of bicategories can be modelled by $2$-category theory, which Lack \cite[\S 1.1]{MR2664622} describes as follows:
\begin{quote}
$2$-category theory is a ``middle way'' between $\Cat$-category theory and bicategory theory. It \emph{uses} enriched category theory, but not in the simple minded way of $\Cat$-category theory; and it cuts through some of the technical nightmares of bicategories.
\end{quote}
This could also be described as ``homotopy coherent $\Cat$-enriched category theory'' (cf.\ \cite{MR1376543}), that is category theory enriched over the base $\Cat$ not merely as a monoidal category, but as a monoidal category with inherent higher structure, as might be realised by considering $\Cat$ as a monoidal model category.

For example, consider the bicategorical analogue of the presheaf category over a small \linebreak$2$-category $\mathcal{C}$, which is the $2$-category $\Hom(\mathcal{C}^\text{op},\Cat)$ whose objects are pseudofunctors from $\mathcal{C}^\text{op}$ to $\Cat$ and whose morphisms are pseudonatural transformations between them. It follows from two-dimensional monad theory \cite{MR1007911,MR2369168} that this $2$-category is related to the $\Cat$-enriched presheaf category $[\mathcal{C}^\text{op},\Cat]$, whose objects and morphisms are $2$-functors and $2$-natural transformations, by a $2$-adjunction
\begin{equation*} 
\xymatrix{
[\mathcal{C}^\text{op},\Cat] \ar@<-1.5ex>[rr]^-{\hdash}_-{} && \ar@<-1.5ex>[ll]_-{} \Hom(\mathcal{C}^\text{op},\Cat)
}
\end{equation*}
that restricts to a biequivalence between $\Hom(\mathcal{C}^\text{op},\Cat)$ and the full sub-$2$-category of  $[\mathcal{C}^\text{op},\Cat]$ on the cofibrant objects for the projective model structure. This $2$-adjunction is used to show that $\Cat$-enriched (co)limits weighted by projective cofibrant weights model bicategorical (co)limits \cite{MR0401868,MR998024,MR2431638}.

One dimension higher, the coherence theorem of Gordon, Power, and Street \cite{MR1261589} states that every tricategory is triequivalent to a $\Gray$-category (that is, a category enriched over Gray's symmetric monoidal closed structure on the category $\twocat$ of $2$-categories and $2$-functors \cite{MR0371990,MR0412252}). Moreover, the \emph{category theory} of tricategories can be modelled by ``homotopy coherent $\Gray$-enriched category theory''. (Note that $\twocat$ is a monoidal model category with the Gray monoidal structure \cite{MR1931220,MR2138540}.) 
However, the relationship between these two theories is more complicated than the relationship one dimension lower. For whereas $\Cat$-category theory and bicategory theory share the common base $\Cat$, the bases of $\Gray$-category theory and tricategory theory, i.e.\ the $\Gray$-category of $2$-categories and the tricategory of bicategories, are distinct \cite{MR2276246}.
Hence a study of the relationship between these two theories must involve a study of the relationship between their bases. 

The underlying categories of these two bases are related by the \emph{strictification adjunction},
\begin{equation} \label{stadj}
\cd{
\twocat \ar@<-1.5ex>[rr]^-{\hdash}_-{} && \ar@<-1.5ex>[ll]_-{\st} \Bicat
}
\end{equation}
whose right adjoint is the inclusion of $\twocat$ into the category $\Bicat$ of bicategories and pseudofunctors, and whose left adjoint sends a bicategory to its strictification (to which it is biequivalent) \cite[\S 4.10]{MR1261589}.  In this paper we study the higher structure of this adjunction. Our fundamental result (Corollary \ref{3duniprop2}), from which our main theorem (stated below) follows immediately, is the \emph{three-dimensional universal property of strict\-ification}, which states that for every bicategory $A$ and $2$-category $B$, the hom-set bijection $\twocat(\st A,B) \cong \Bicat(A,B)$ of the strictification adjunction underlies an isomorphism of $2$-categories
\begin{equation} \label{funiso}
\mathbf{Ps}(\st A,B) \cong \mathbf{Hom}(A,B),
\end{equation}
where $\mathbf{Ps}(-,-)$ and $\mathbf{Hom}(-,-)$ denote the homs of the $\Gray$-category of $2$-categories and the tricategory of bicategories respectively (whose morphisms are $2$-functors and pseudofunctors respectively, and whose $2$-cells and $3$-cells are in both cases pseudonatural transformations and modifications).

This higher universal property suggests that the strictification adjunction underlies some kind of ``three-dimensional adjunction''. Working in the setting of tricategory theory, one can indeed show that the strictification adjunction underlies a triadjunction between the tricategories of $2$-categories and bicategories. Note that to have such a triadjunction it would suffice for (\ref{funiso}) to be a biequivalence; yet despite the strength of (\ref{funiso}) being in fact an isomorphism, the ``weakness'' provided by the definitions of tricategory theory is still required in order to realise the higher structure of the strictification adjunction in this setting, for neither the tricategory of bicategories nor the strictification trihomomorphism is strict.

In this paper, however, we work within an alternative framework for bicategory-enriched categories,\footnote{By which we mean enriched categories whose hom-objects are bicategories, and not categories enriched over a bicategory in the sense of \cite{MR647583,MR2219705}.} in which we show that the same three-dimensional higher structure of the strictification adjunction can be realised by a \emph{strictly} bicategory-enriched adjunction. This is the framework of enrichment over the \emph{symmetric closed multicategory} ${\sf Bicat}$ of bicategories introduced by Verity \cite[\S 1.3]{MR2844536}. Note that in this setting the category of bicategories is a \emph{strictly} bicategory-enriched category, in contrast to the setting of tricategory theory where it is merely ``weakly'' enriched.

We summarise the central argument of this paper as follows. Using standard arguments of enriched category theory generalised in \S \ref{adjmultsect} to the context of enrichment of and over symmetric multicategories, in \S \ref{stadjsect} we prove our main theorem (Theorem \ref{fincor}) as a formal consequence of the three-dimensional universal property of strictification (\ref{funiso}): 
\begin{theorem*}
The strictification adjunction {\normalfont(\ref{stadj})} underlies an adjunction of ${\sf Bicat}$-enriched categories and, moreover,  an adjunction of ${\sf Bicat}$-enriched symmetric multicategories.
\end{theorem*}

As an application of the strictness of strictification revealed by this main theorem, we obtain (Proposition \ref{newgray}) a hitherto undiscovered \emph{$\Gray$-category of bicategories}, whose underlying category is the category $\Bicat$ of bicategories and pseudofunctors, and whose hom $2$-categories $\st\,\Hom(A,B)$ are the strict\-ifications of the hom bicategories $\Hom(A,B)$. This $\Gray$-category is triequivalent (via a bijective-on-objects, bijective-on-morphisms trihomomorphism) to the tricategory of bicategories.

For the remainder of this section, let us explain in detail why it is that the category of bicategories (and hence ultimately the strictification adjunction) can be strictly enriched over the multicategory ${\sf Bicat}$, while the tricategory of bicategories fails to be strict. The reason is that composition, specifically \emph{horizontal composition of $2$-cells}, is encoded in different ways in tricategories and ${\sf Bicat}$-enriched categories. This difference is a result of the difference between the cartesian monoidal category $\Bicat$ and the symmetric multicategory  ${\sf Bicat}$, for in a tricategory, composition is given by pseudofunctors out of cartesian products of bicategories as on the left below,
\begin{equation*}
\Hom(B,C) \times \Hom(A,B) \lra \Hom(A,C) \qquad (\Hom(B,C),\Hom(A,B)) \lra \Hom(A,C)
\end{equation*} 
whereas in a ${\sf Bicat}$-enriched category, composition is given by ``two-variable pseudofunctors'', i.e.\ binary morphisms in the multicategory ${\sf Bicat}$, as on the right above; more generally, composition in a category enriched over a multicategory is given by such binary morphisms \cite{MR0242637,MR0288162}. (Note that we use different fonts to distinguish between these two structures: the boldface $\Bicat$ denotes the cartesian monoidal category, whereas the sans-serif $\sf{Bicat}$ denotes the multicategory.)

Recall that in a $2$-category (and more generally in a bicategory) the operation of horizontal composition of $2$-cells can be derived from the operations of vertical composition of $2$-cells and whiskering of $2$-cells by morphisms on either side, as in the equations displayed below.
\begin{equation} \label{cinq}
\cd{
A \ar@/^.75pc/[r]^-f  \ar@/_.75pc/[r]_-g \dtwocell{r}{\alpha} & B \ar[r]^-h \ar@{}[d]|-{\circ} & C\\
A \ar[r]_-g & B \ar@/^.75pc/[r]^-h  \ar@/_.75pc/[r]_-k \dtwocell{r}{\beta} & C
}
\quad
=
\quad
\cd{
A \ar@/^.75pc/[r]^-f  \ar@/_.75pc/[r]_-g \dtwocell{r}{\alpha} & B \ar@/^.75pc/[r]^-h  \ar@/_.75pc/[r]_-k \dtwocell{r}{\beta} & C
}
\quad
=
\quad
\cd{
A \ar[r]^-f & B \ar@/^.75pc/[r]^-h  \ar@/_.75pc/[r]_-k \dtwocell{r}{\beta} \ar@{}[d]|-{\circ} & C \\
A \ar@/^.75pc/[r]^-f  \ar@/_.75pc/[r]_-g \dtwocell{r}{\alpha} & B \ar[r]_-k & C
}
\end{equation}
Hence a $2$-category can be defined in such a way that vertical composition of $2$-cells and whiskering are primitive operations satisfying Godement's \emph{cinq r{\`e}gles} \cite[Appendice \S1]{MR0102797}, one of which states that the left-hand and right-hand sides of (\ref{cinq}) are equal, and such that horizontal composition of $2$-cells is a derived operation defined by (\ref{cinq}) (see for instance \cite[\S 2]{MR1421811}). This amounts to defining a $2$-category as a category enriched over the multicategory $\sf{Cat}$ of categories, in which a binary morphism $F \colon (A,B) \lra C$ consists of:
\begin{enumerate}[(i)]
\item a function $F \colon \mathrm{ob}A \times \mathrm{ob} B \lra \mathrm{ob}C$,
\item for each object $a \in A$, a functor $F(a,-) \colon B \lra C$ agreeing with the function (i) on objects,
\item for each object $b \in B$, a functor $F(-,b) \colon A \lra C$ agreeing with the function (i) on objects,
\end{enumerate}
\nopagebreak[2]
such that for each pair of morphisms $f \colon a \lra a'$ in $A$ and $g \colon b \lra b'$ in $B$, the following square commutes.
\begin{equation*} 
\cd[@C=2.5em]{
F(a,b) \ar[r]^-{F(a,g)} \ar[d]_-{F(f,b)} \ar@{}[dr]|-{=} & F(a,b') \ar[d]^-{F(f,b')} \\
F(a',b) \ar[r]_-{F(a',g)} & F(a',b')
}
\end{equation*}
It is well known that such binary morphisms are in bijection with functors $F \colon A \times B \lra C$ from the cartesian product of categories (see \cite[Theorem 5.2]{MR0013131}), and indeed the cartesian monoidal category $\Cat$ \emph{represents} the multicategory $\sf{Cat}$ (in the sense of \cite{MR1758246}). Hence a $2$-category can be equivalently defined either as a category enriched over the cartesian monoidal category $\Cat$ or as a category enriched over the multicategory $\sf{Cat}$, in which cases horizontal composition of $2$-cells is either a  primitive or a derived operation. 

Similarly, in a category enriched over the cartesian monoidal category $\Bicat$, horizontal composition of $2$-cells is a primitive operation, whereas in a category enriched over Verity's multicategory $\sf{Bicat}$ it is not, and indeed in a $\sf{Bicat}$-enriched category there is no canonical operation of horizontal composition of $2$-cells. For a binary morphism $F \colon (A,B) \lra C$ in ${\sf Bicat}$ consists of:
\begin{enumerate}[(i)]
\item a function $F \colon \mathrm{ob}A \times \mathrm{ob} B \lra \mathrm{ob}C$,
\item for each object $a \in A$, a pseudofunctor $F(a,-) \colon B \lra C$ agreeing with the function (i) on objects,
\item for each object $b \in B$, a pseudofunctor $F(-,b) \colon A \lra C$ agreeing with the function (i) on objects,
\item for each pair of morphisms $f \colon a \lra a'$ in $A$ and $g \colon b \lra b'$ in $B$, an invertible $2$-cell $F(f,g)$ in $C$ as displayed below,
\begin{equation*} 
\cd[@C=2.5em]{
F(a,b) \ar[r]^-{F(a,g)} \ar[d]_-{F(f,b)} \dtwocell{dr}{F(f,g)} & F(a,b') \ar[d]^-{F(f,b')} \\
F(a',b) \ar[r]_-{F(a',g)} & F(a',b')
}
\end{equation*}
\end{enumerate}
subject to axioms (formally identical, when presented as pasting equations, to those of a cubical functor of two variables; see for instance \cite[\S4.2]{MR1261589}).
Hence for each ``horizontally composable'' pair of $2$-cells $\alpha$ and $\beta$ in a $\sf{Bicat}$-category, there is a specified invertible $3$-cell 
\begin{equation} \label{cinqgray}
\cd{
A \ar@/^.75pc/[r]^-f  \ar@/_.75pc/[r]_-g \dtwocell{r}{\alpha} & B \ar[r]^-h \ar@{}[d]|-{\circ} & C\\
A \ar[r]_-g & B \ar@/^.75pc/[r]^-h  \ar@/_.75pc/[r]_-k \dtwocell{r}{\beta} & C
}
\quad
\cong
\quad
\cd{
A \ar[r]^-f & B \ar@/^.75pc/[r]^-h  \ar@/_.75pc/[r]_-k \dtwocell{r}{\beta} \ar@{}[d]|-{\circ} & C \\
A \ar@/^.75pc/[r]^-f  \ar@/_.75pc/[r]_-g \dtwocell{r}{\alpha} & B \ar[r]_-k & C
}
\end{equation}
in place of the equality of Godement's rule (\ref{cinq}), and neither its source nor its target $2$-cell has a stronger claim than the other to be the horizontal composite of $\alpha$ and $\beta$; note that the situation is precisely the same in a $\Gray$-category. This phenomenon can be observed in the composition of pseudonatural transformations: given a pair of  pseudonatural transformations $\alpha$ and $\beta$ as in (\ref{cinqgray}), there is a canonical invertible modification $\beta g \circ h \alpha \lra k \alpha \circ \beta f$ whose component at an object $a \in A$ is given by the pseudonaturality constraint for $\beta$ at the morphism $\alpha_a \colon fa \lra ga$ of $B$, as displayed below.
\begin{equation}  \label{canmod}
\cd{
hfa \ar[r]^-{h\alpha_a} \ar[d]_-{\beta_{fa}} \dtwocell{dr}{\beta_{\alpha_a}} & hga \ar[d]^-{\beta_{ga}} \\
kfa \ar[r]_-{k\alpha_a} & kga
}
\end{equation}

Therefore to define an enrichment of the category of bicategories over the cartesian monoidal category $\Bicat$ with enriched hom-objects $\Hom(A,B)$, a definition of horizontal composition for pseudonatural transformations must be chosen, and this will inevitably fail to be strictly associative (see \cite[Proposition 5.3]{MR3076451} for an invertible icon witnessing this failure); one has at best a weak enrichment over the cartesian monoidal $2$-category $\Bicat_2$ of bicategories, pseudofunctors, and  icons (see \cite[\S6]{MR2871166} and \cite[\S4.1]{MR3294506}). On the other hand, since horizontal composition of $2$-cells is not a primitive operation in a $\sf{Bicat}$-category, no such choice need be made in order to define a ${\sf Bicat}$-enrichment of the category of bicategories, and indeed the desired (strict!) enrichment follows formally from the existence of the closed structure of the multicategory $\sf{Bicat}$.

Similar remarks can be made about the higher structure of the strictification functor  $\st \colon \Bicat \lra \twocat$. An extension of the strictification functor  to a trihomomorphism between tricategories was defined in \cite[\S5.6]{MR1261589} and in \cite[\S8.2]{MR3076451}, and by the latter definition it is clear that $\st$ can be extended to a $\Bicat_2$-enriched pseudofunctor from the $\Bicat_2$-enriched bicategory of bicategories mentioned in the previous paragraph; note that neither its unit constraint nor its composition constraint is an identity. However, as we shall prove in Theorem \ref{fincor}, the strictification functor can be \emph{strictly} enriched over the multicategory $\sf{Bicat}$. Once again, this difference of behaviour is due to the different encodings of horizontal composition of $2$-cells. For a $\Bicat$-enriched functor must preserve horizontal composition of $2$-cells and the identity $2$-cells of identity morphisms, which $\st$ fails to do, whereas a $\sf{Bicat}$-enriched functor need only preserve the canonical isomorphisms of (\ref{cinqgray}), which $\st$ does.

In summary, although the standard definitions of tricategory theory (see for instance \cite{MR1261589,MR3076451}) use the cartesian product of bicategories,  we argue that -- in order to work ``as strictly as possible'' (cf.\ \cite{MR1935980}) -- Verity's multicategory is the more suitable base of enrichment for studying the higher structure of the category of bicategories and the strictification adjunction. For whereas the category of bicategories and the strictification functor are only ``weakly'' enriched over the cartesian monoidal structure, we show that they can be \emph{strictly} enriched over the multicategory structure. Moreover, whereas the strictification functor is only ``weakly'' monoidal with respect to the cartesian product of bicategories and the symmetric Gray tensor product of \linebreak$2$-categories (see \cite{MR3023916}), we show that it underlies a (strict) multifunctor from the multicategory ${\sf Bicat}$ of bicategories to the multicategory ${\sf Gray}$ of $2$-categories represented by the symmetric Gray monoidal structure.

The source of this difference is the failure of the cartesian monoidal structure to represent the multicategory structure; in particular, pseudofunctors $A \times B \lra C$ are not generally in bijection with two-variable pseudofunctors $(A,B) \lra C$. In Appendix \ref{appendix} we reconcile these two structures by showing that there is an enrichment of $\sf{Bicat}$ to a $2$-multicategory, whose $2$-cells are multivariable icons, which is \emph{birepresented} (i.e.\ represented up to equivalence) by the cartesian monoidal $2$-category $\Bicat_2$ (see Theorem \ref{birep}).
For this and other purposes it is useful to work with the larger symmetric closed multicategory $\sf{PsDbl}$ of pseudo double categories, which contains $\sf{Bicat}$ as a full sub-multicategory, and which we introduce in Appendix \ref{appendix}. Nevertheless, our primary objects of interest remain bicategories; the main benefit of working at this greater level of generality is natural access to the hom pseudo double categories $\underline{\Hom}(A,B)$, which supplement the usual hom bicategories $\Hom(A,B)$ with the ever-useful icons \cite{MR2640216}.

\subsection*{Acknowledgements}
The support of Australian Research Council Future Fellowship \linebreak FT160100393 is gratefully acknowledged. The author thanks Richard Garner and an anonymous referee for their helpful comments on earlier drafts of this paper.

\section{Adjunction for enriched multicategories} \label{adjmultsect}
The purpose of this section is to provide the means by which we will deduce the enrichment of the strictification adjunction (\ref{stadj}) over the symmetric multicategory of bicategories from the three-dimensional universal property of strictification (\ref{funiso}). The main results of this section are all instances of the fundamental categorical principle that \emph{universality begets functoriality}, whose most basic instance is the standard result that a functor $T \colon \mathcal{A} \lra \mathcal{B}$ has a left adjoint if and only if the functor $\mathcal{B}(B,T-) \colon \mathcal{A} \lra \Set$ is representable for each object $B \in \mathcal{B}$. We first generalise this standard result to adjunctions of enriched symmetric multicategories (Lemma \ref{adjfunlemma}), and then apply this generalisation to prove, for each symmetric multifunctor $T \colon {\sf V} \lra {\sf W}$ between symmetric closed multicategories, a necessary and sufficient condition (Theorem \ref{adjfunthm}) for  the induced symmetric ${\sf W}$-multifunctor $\widehat{T} \colon T_{\ast}\underline{\sf V} \lra \underline{\sf W}$ to have a left adjoint. In the following section we will apply this theorem to the inclusion ${\sf Gray} \lra {\sf Bicat}$ of the symmetric closed multicategory of $2$-categories (represented by the symmetric closed monoidal category $\Gray$) into the symmetric closed multicategory of bicategories to prove our main theorem. To begin, we recall the relevant basic theory of enrichment of and over symmetric multicategories.

Although enrichment over a symmetric multicategory generalises enrichment over a symmetric monoidal category, one can work over a symmetric multicategory as one does over a symmetric strict monoidal category  with the help of the following construction. The \emph{symmetric monoidal envelope}\footnote{This name is given in \cite[\S2.2.4]{lurieha} to the corresponding $\infty$-categorical construction.} of a symmetric multicategory ${\sf V}$ is the symmetric strict monoidal category $\mathbb{F}({\sf V})$ whose objects $(X_1,\ldots,X_n)$ are words of objects of ${\sf V}$ of length $n \geq 0$, and in which a morphism $(X_1,\ldots,X_n) \lra (Y_1,\ldots,Y_m)$ consists of a function $\varphi \colon \{1,\ldots,n\} \lra \{1,\ldots,m\}$ and a multimorphism $(X_i)_{i \in \varphi^{-1}(j)} \lra Y_j$ in ${\sf V}$ for each $1 \leq j \leq m$; composition in $\mathbb{F}({\sf V})$ is defined using the symmetric multicategory structure of ${\sf V}$, and the tensor product and symmetry are given by concatenation and permutation of words. This construction defines a functor from the category of symmetric multicategories and symmetric multifunctors to the category of symmetric strict monoidal categories (a.k.a.\ permutative categories) and symmetric strict monoidal functors, which is left adjoint to the functor that sends a symmetric strict monoidal category $\mathcal{V}$ to the symmetric multicategory with the same objects as $\mathcal{V}$ and  whose multimorphisms $(X_1,\ldots,X_n) \lra Y$ are morphisms $X_1 \otimes \cdots \otimes X_n \lra Y$ in $\mathcal{V}$ \cite[Proposition 4.2]{MR2558315}.

The non-symmetric version of this construction, namely the monoidal envelope of a multicategory, was defined in \cite{MR0288162}, where it was used (among other things) to define the notions of category, functor, and natural transformation enriched over a multicategory in terms of the same notions enriched over a strict monoidal category. We make the analogous definitions in the symmetric case (cf.\ \cite[\S 3]{MR2523451}). 

\begin{definition}
Let ${\sf V}$ be a symmetric multicategory. A \emph{${\sf V}$-category}  is an $\mathbb{F}({\sf V})$-category $\mathcal{A}$ each of whose hom-objects $\mathcal{A}(A,B) \in \mathbb{F}({\sf V})$ is a word of length $1$. The $2$-category ${\sf V}\text{-}\Cat$ of  ${\sf V}$-categories, \emph{${\sf V}$-functors}, and \emph{$\sf{V}$-natural transformations} is the full sub-$2$-category of $\mathbb{F}({\sf V})\text{-}\Cat$ on the ${\sf V}$-categories.
An \emph{adjunction of ${\sf V}$-categories} is an adjunction in the $2$-category ${\sf V}\text{-}\Cat$.

A \emph{symmetric ${\sf V}$-multicategory} is a symmetric $\mathbb{F}({\sf V})$-multicategory ${\sf A}$ each of whose hom-objects ${\sf A}(A_1,\ldots,A_n;B) \in \mathbb{F}({\sf V})$ is a word of length $1$ (for $n \geq 0$). The $2$-category  ${\sf V}\text{-}\mathbf{SMult}$ of symmetric  ${\sf V}$-multicategories, \emph{symmetric ${\sf V}$-multifunctors}, and \emph{${\sf V}$-multinatural transformations} is the full sub-$2$-category of $\mathbb{F}({\sf V})\text{-}\mathbf{SMult}$ on the symmetric ${\sf V}$-multicategories. 
An \emph{adjunction of symmetric ${\sf V}$-multicategories} is an adjunction in the $2$-category ${\sf V}\text{-}\mathbf{SMult}$.
\end{definition}

Since each symmetric multifunctor $T \colon {\sf V} \lra {\sf W}$ induces a symmetric strict monoidal functor $\mathbb{F}(T) \colon \mathbb{F}({\sf V}) \lra \mathbb{F}({\sf W})$ that preserves the lengths of objects, the change of base $2$-functor along $\mathbb{F}(T)$ restricts to a $2$-functor $T_{\ast} \colon {\sf V}\text{-}\mathbf{SMult} \lra {\sf W}\text{-}\mathbf{SMult}$, which we call the \emph{change of base  $2$-functor along $T$}. Note that this $2$-functor sends adjunctions of symmetric $\sf {V}$-multicategories to adjunctions of symmetric ${\sf W}$-multicategories. Explicitly, for each symmetric ${\sf V}$-multicategory ${\sf A}$, $T_{\ast}{\sf A}$ is a symmetric ${\sf W}$-multicategory with the same objects as ${\sf A}$ and with hom-sets  $(T_{\ast}{\sf A})(A_1,\ldots,A_n;B) = T{\sf A}(A_1,\ldots,A_n;B)$.
 In particular, change of base along the symmetric multifunctor $V \colon {\sf V} \lra {\sf Set}$ that sends an object $X$ of ${\sf V}$ to its underlying set $VX = {\sf V}(\,\,;X)$ defines a $2$-functor $V_{\ast} \colon {\sf V}\text{-}\mathbf{SMult} \lra \mathbf{SMult}$ that sends a symmetric ${\sf V}$-multicategory to its underlying symmetric multicategory.

Each symmetric closed multicategory ${\sf V}$, with internal hom objects $[X,Y]$, admits a canonical self-enrichment  to a symmetric ${\sf V}$-multicategory $\underline{\sf V}$, whose objects are those of ${\sf V}$ and whose hom-objects are defined recursively as $\underline{{\sf V}}(\,\,;Y) = Y$ for $n=0$, and as 
\begin{equation*}
\underline{{\sf V}}(X_1,\ldots,X_n;Y) = \underline{{\sf V}}(X_1,\ldots,X_{n-1};[X_n,Y])
\end{equation*}
 for $n \geq 1$. Moreover,  each symmetric multifunctor $T \colon {\sf V} \lra {\sf W}$ between symmetric closed multicategories induces a symmetric ${\sf W}$-multifunctor $\widehat{T} \colon  T_{\ast}\underline{\sf V} \lra \underline{\sf W}$ that agrees with $T$ on objects and is defined on hom-objects recursively as the identity on $T\underline{{\sf V}}(\,\,;X) = TX = \underline{{\sf W}}(\,\,;TX)$ for $n=0$, and as the composite
\footnotesize
\begin{equation*}
\cd[@C=2em]{
T\underline{{\sf V}}(X_1,\ldots,X_{n-1};[X_n,Y]) \ar[r]^-{\widehat{T}} & \underline{{\sf W}}(TX_1,\ldots,TX_{n-1};T[X_n,Y]) \ar[rr]^-{\underline{{\sf W}}((1);\psi)} && \underline{{\sf W}}(TX_1,\ldots,TX_{n-1};[TX_n,TY])
}
\end{equation*}
\normalsize
for $n \geq 1$,
where $\psi \colon T[X,Y] \lra [TX,TY]$ corresponds under the canonical bijection 
\begin{equation*}
{\sf W}(T[X,Y];[TX,TY]) \cong {\sf W}(T[X,Y],TX;TY)
\end{equation*}
to the image under $T$ of the evaluation morphism $\mathrm{ev} \colon ([X,Y],X) \lra Y$ in ${\sf V}$. (See \cite[Chapter 4]{MR2513383} for further details of these constructions.)

The proof of the main theorem of this section (Theorem \ref{adjfunthm}) uses the following lemma, which generalises to adjunctions of enriched symmetric multicategories the standard categorical result  (see for instance \cite[Theorem IV.1.2]{MR1712872}) that to give a left adjoint to a functor $T \colon \mathcal{A} \lra \mathcal{B}$ is precisely to give, for each object $B \in \mathcal{B}$, a representation of the functor $\mathcal{B}(B,T-) \colon \mathcal{A} \lra \Set$. The latter amounts by the Yoneda lemma to an object $SB \in \mathcal{A}$ and a morphism $\eta_B \colon B \lra TSB$ in $\mathcal{B}$ with the universal property that the composite function
\begin{equation*}
\cd[@=3em]{
\mathcal{A}(SB,A) \ar[r]^-T & \mathcal{B}(TSB,TA) \ar[r]^-{\mathcal{B}(\eta_B,1)} & \mathcal{B}(B,TA)
}
\end{equation*}
is a bijection for each object $A \in \mathcal{A}$.

\begin{lemma} \label{adjfunlemma}
Let ${\sf V}$ be a symmetric multicategory and let $T \colon \sf{A} \lra \sf{B}$ be a symmetric  \linebreak$\sf{V}$-multifunctor between symmetric $\sf{V}$-multicategories. If for each object $B\in \sf{B}$ there exists an object $SB\in \sf{A}$ and a morphism $\eta_B \colon B \lra TSB$ in $\sf{B}$ such that the composite morphism 
\begin{equation} \label{kellyn}
\cd[@C=2.5em]{
{\sf A} (SB_1,\dots,SB_n;A) \ar[r]^-T & {\sf B}(TSB_1,\ldots,TSB_n;TA) \ar[rr]^-{{\sf B}(\eta_{B_1},\ldots,\eta_{B_n};1)} && {\sf B}(B_1,\ldots,B_n;TA)
}
\end{equation}
is an isomorphism in ${\sf V}$ for each $n \geq 0$ and objects $B_1,\ldots,B_n \in \sf{B}$ and $A \in \sf{A}$, then this data extends uniquely to an adjunction of symmetric $\sf{V}$-multicategories $S \dashv T \colon \sf{A} \lra \sf{B}$ with unit $\eta$.
\end{lemma}
\begin{proof}
Let $N = N_{B_1,\ldots,B_n;A}$ denote the invertible morphism (\ref{kellyn}). We define the action of the symmetric ${\sf V}$-multifunctor $S \colon {\sf B} \lra {\sf A}$ on hom-objects to be the following composite morphism.
\begin{equation*}
\cd[@C=2.75em]{
{\sf B}(B_1,\ldots,B_n;C) \ar[r]^-{{\sf B}((1);\eta)} & {\sf B}(B_1,\ldots,B_n;TSC) \ar[r]^-{N^{-1}} & {\sf A}(SB_1,\ldots,SB_n;SC)
}
\end{equation*}
Preservation of composition by $S$ is proved by the commutativity of the following diagram in the symmetric monoidal envelope $\mathbb{F}({\sf V})$,
\tiny
\begin{equation*}
\cd[@C=3.29em]{
{\sf B}((C_k);D) \otimes \bigotimes_{k} {\sf B}((B_j^k);C_k) \ar[r]^-{\circ} \ar[d]_-{{\sf B}((1);\eta) \otimes 1} & {\sf B}((B_j^k);D) \ar[d]^-{{\sf B}((1);\eta)} \\
{\sf B}((C_k);TSD) \otimes \bigotimes_{k} {\sf B}((B_j^k);C_k) \ar[r]^-{\circ} \ar[dd]_-{N^{-1} \otimes 1} & {\sf B}((B_j^k);TSD) \ar[r]^-{N^{-1}} & {\sf A}((SB_j^k);SD) \\
& {\sf B}((TSC_k);TSD) \otimes \bigotimes_k {\sf B}((B_j^k);TSC_k) \ar[u]_-{\circ} \\
{\sf A}((SC_k);SD) \otimes \bigotimes_k {\sf B}((B_j^k);C_k) \ar[r]_-{1 \otimes \bigotimes {\sf B}((1);\eta)} & {\sf A}((SC_k);SD) \otimes\bigotimes_k {\sf B}((B_j^k);TSC_k) \ar[u]_-{T \otimes 1} \ar[r]_-{1 \otimes \bigotimes N^{-1}} & {\sf A}((SC_k);SD) \otimes\bigotimes_k {\sf A}((SB_j^k);SC_k) \ar[uu]_-{\circ}
}
\end{equation*}
\normalsize
whose upper region commutes by associativity of composition in ${\sf B}$, whose lower left region commutes by the diagram
\begin{equation*}
\cd[@C=5em]{
{\sf B}((C_k);TSD) \otimes \bigotimes_{k} {\sf B}((B_j^k);C_k) \ar[r]^-{\circ} & {\sf B}((B_j^k);TSD) \\
{\sf B}((TSC_k);TSD) \otimes \bigotimes_{k} {\sf B}((B_j^k);C_k) \ar[u]^-{{\sf B}((\eta);1) \otimes 1} \ar[r]^-{1 \otimes \bigotimes {\sf B}((1);\eta)} & {\sf B}((TSC_k);TSD) \otimes \bigotimes_k {\sf B}((B_j^k);TSC_k) \ar[u]_-{\circ} \\
{\sf A}((SC_k);SD) \otimes \bigotimes_k {\sf B}((B_j^k);C_k) \ar[r]_-{1 \otimes \bigotimes {\sf B}((1);\eta)} \ar[u]^-{T \otimes 1} & {\sf A}((SC_k);SD) \otimes\bigotimes_k {\sf B}((B_j^k);TSC_k) \ar[u]_-{T \otimes 1}
}
\end{equation*}
and whose lower right region commutes by the diagram
\footnotesize
\begin{equation*}
\cd{
{\sf B}((B_j^k);TSD) & {\sf B}((TSB_j^k);TSD) \ar[l]_-{{\sf B}((\eta);1)} & {\sf A}((SB_j^k);SD) \ar[l]_-{T} \\
\bullet \ar[u]^-{\circ} & {\sf B}((TSC_k);TSD) \otimes \bigotimes_k {\sf B}((TSB_j^k);TSC_k) \ar[u]_-{\circ} \ar[l]_-{1 \otimes \bigotimes {\sf B}((\eta);1)} \\
\bullet \ar[u]^-{T \otimes 1} & {\sf A}((SC_k);SD) \otimes \bigotimes_k {\sf B}((TSB_j^k);TSC_k) \ar[u]_-{T \otimes 1} \ar[l]^-{1 \otimes \bigotimes {\sf B}((\eta);1)} & {\sf A}((SC_k);SD) \otimes \bigotimes_k {\sf A}((SB_j^k);SC_k) \ar[uu]_-{\circ} \ar[l]^-{1 \otimes \bigotimes T}
}
\end{equation*}
\normalsize
where  the latter two diagrams commute by associativity of composition in ${\sf B}$, by functoriality of the tensor product in $\mathbb{F}({\sf V})$, and by ${\sf V}$-multifunctoriality of $T$.

Preservation of identities by $S$ is proved by the commutativity of the following diagram,
\begin{equation*}
\cd{
& I \ar[r]^-j \ar[dl]_-j \ar[d]^-j & {\sf B}(B;B) \ar[d]^-{{\sf B}(1;\eta)} \\
{\sf A}(SB;SB) \ar[r]_-T & {\sf B}(TSB;TSB) \ar[r]_-{{\sf B}(\eta;1)} & {\sf B}(B;TSB)
}
\end{equation*}
which commutes by the identity axioms for composition in ${\sf B}$ and since $T$ preserves identities.

Equivariance of $S$ is proved by the commutativity of the following diagram,
\begin{equation*}
\cd[@C=3em]{
{\sf B}((B_i);C) \ar[r]^-{{\sf B}((1);\eta)} \ar[d]_-{\sigma} & {\sf B}((B_i);TSC) \ar[r]^-{N^{-1}} \ar[d]^-{\sigma} & {\sf A}((SB_i);SC) \ar[d]^-{\sigma} \\
{\sf B}((B_{\sigma(i)});C) \ar[r]_-{{\sf B}((1);\eta)} & {\sf B}((B_{\sigma(i)});TSC) \ar[r]_-{N^{-1}} & {\sf A}((SB_{\sigma(i)});SC)
}
\end{equation*}
whose left-hand region commutes by equivariance of composition in ${\sf B}$ and whose right-hand region commutes by the commutativity of the diagram
\begin{equation*}
\cd[@C=3em]{
{\sf B}((B_i);TSC) \ar[d]_-{\sigma} & {\sf B}((TSB_i);TSC) \ar[d]^-{\sigma} \ar[l]_-{{\sf B}((\eta);1)} & {\sf A}((SB_i);SC) \ar[d]^-{\sigma} \ar[l]_-T \\
{\sf B}((B_{\sigma(i)});TSC) & {\sf B}((TSB_{\sigma(i)});TSC) \ar[l]^-{{\sf B}((\eta);1)} & {\sf A}((SB_{\sigma(i)});SC) \ar[l]^-{T}
}
\end{equation*}
which commutes by equivariance of composition in ${\sf B}$ and by equivariance of $T$.

The ${\sf V}$-multinaturality of $\eta \colon 1 \lra TS$ is expressed by the commutativity of the diagram
\begin{equation*}
\cd{
{\sf B}(B_1,\ldots,B_n;C) \ar[r]^-{{\sf B}((1);\eta)} \ar[d]_-S & {\sf B}(B_1,\ldots,B_n;TSC) \\
{\sf A}(SB_1,\ldots,SB_n;SC) \ar[r]_-T & {\sf B}(TSB_1,\ldots,TSB_n;TSC) \ar[u]_-{{\sf B}(\eta,\ldots,\eta;1)}
}
\end{equation*}
which commutes by definition of the ${\sf V}$-multifunctor $S$.

The definition and ${\sf V}$-multinaturality of the counit $\varepsilon$ are given by the Yoneda lemma for enriched symmetric multicategories. Explicitly, define $\varepsilon_A \in V{\sf A}(STA;A)$ to be the unique element corresponding under the bijection $VN \colon V{\sf A}(STA;A) \cong V{\sf B}(TA;TA)$ to the identity $1_{TA}$; that is, $\varepsilon_A \colon STA \lra A$ is the unique morphism of ${\sf A}$ such that $T\varepsilon_A \circ \eta_{TA} = 1_{TA}$.  The  ${\sf V}$-multinaturality of $\varepsilon$ is expressed by the commutativity of the following diagram,
\begin{equation*}
\cd{
{\sf B}(TA_1,\ldots,TA_n;TB) \ar[r]^-S & {\sf A}(STA_1,\ldots,STA_n;STB) \ar[d]^-{{\sf A}((1);\varepsilon)} \\
{\sf A}(A_1,\ldots,A_n;B) \ar[u]^-T \ar[r]_-{{\sf A}((\varepsilon);1)} & {\sf A}(STA_1,\ldots,STA_n;B)
}
\end{equation*}
which commutes by the invertibility of $N$ (\ref{kellyn}) applied to the diagram
\scriptsize
\begin{equation*}
\cd{
{\sf A}((A_i);B) \ar[r]^-T \ar[d]_-{{\sf A}((\varepsilon);1)} & {\sf B}((TA_i);TB) \ar[d]_-{{\sf B}((T\varepsilon);1)} \ar[r]^-{{\sf B}((1);\eta)} \ar[dr]^-1& {\sf B}((TA_i);TSTB) \ar[d]^-{{\sf B}((1);T\varepsilon)} & {\sf B}((TSTA_i);TSTB) \ar[d]^-{{\sf B}((1);T\varepsilon)} \ar[l]_-{{\sf B}((\eta);1)} & {\sf A}((STA_i);STB) \ar[d]^-{{\sf A}((1);\varepsilon)} \ar[l]_-T \\
{\sf A}((STA_i);B) \ar[r]_-T & {\sf B}((TSTA_i);TB) \ar[r]_-{{\sf B}((\eta);1)} & {\sf B}((TA_i);TB) & {\sf B}((TSTA_i);TB) \ar[l]^-{{\sf B}((\eta);1)} & {\sf A}((STA_i);B) \ar[l]^-T
}
\end{equation*}
\normalsize
which commutes by ${\sf V}$-multifunctoriality of $T$, associativity of composition in ${\sf B}$, and the triangle identity defining $\varepsilon$. The other triangle identity follows by the standard argument.
\end{proof} 

We now prove the main theorem of this section, which gives 
a necessary and sufficient condition for the symmetric $\sf{W}$-multifunctor $\widehat{T} \colon T_{\ast}\underline{\sf{V}} \lra \underline{\sf{W}}$ induced by a symmetric multifunctor $T \colon {\sf V} \lra {\sf W}$ between symmetric closed multicategories to be the right adjoint of an adjunction of symmetric $\sf{W}$-multicategories. This theorem generalises the corresponding result for symmetric closed monoidal categories (see \cite[\S 5]{MR0255632} and \cite{MR3455390}).

\begin{theorem} \label{adjfunthm}
Let $T \colon {\sf V}\lra {\sf W}$ be a symmetric multifunctor between symmetric closed multicategories. If for each object $Y \in \sf{W}$ there exists an object $SY\in\sf{V}$ and a morphism $\eta_Y \colon Y \lra TSY$ in $\sf{W}$ such that the composite morphism 
\begin{equation} \label{kellynm}
\cd[@C=3em]{
T[SY,X] \ar[r]^-{\psi} & [TSY,TX] \ar[r]^-{[\eta_Y,1]} & [Y,TX]
}
\end{equation}
is an isomorphism in ${\sf W}$ for each pair of objects $Y \in \sf{W}$ and $X \in \sf{V}$, then  this data extends uniquely to an adjunction of symmetric $\sf{W}$-multicategories $S \dashv \widehat{T}$ whose right adjoint is the symmetric ${\sf W}$-multifunctor $\widehat{T} \colon T_{\ast}\underline{\sf{V}} \lra \underline{\sf{W}}$ induced by $T$ and whose unit is $\eta$.
\end{theorem}
\begin{proof}
We prove by induction on $n \geq 0$ that the symmetric ${\sf W}$-multifunctor $\widehat{T} \colon T_{\ast}\underline{\sf{V}} \lra \underline{\sf{W}}$ induced by $T$ satisfies the hypothesis of Lemma \ref{adjfunlemma}. It is immediate for $n = 0$, since in this case the morphism (\ref{kellyn}) is an identity by definition. For $n \geq 1$, the composite (\ref{kellyn}) is equal by definition to the composite of the upper boundary of the following commutative diagram,

\scriptsize
\begin{equation*}
\cd[@C=1.9em]{
T\underline{\sf V}(SY_1,\ldots,SY_{n-1};[SY_n,X]) \ar[r]^-{\widehat{T}} \ar[dr]_-{\cong} & \underline{\sf W}(TSY_1,\ldots,TSY_{n-1};T[SY_n,X]) \ar[rr]^-{\underline{\sf W}((1);\psi)} \ar[d]^-{\underline{\sf W}(\eta,\ldots,\eta;1)} && \underline{\sf W}(TSY_1,\ldots,TSY_{n-1};[TSY_n,TX]) \ar[d]^-{\underline{\sf W}(\eta,\ldots,\eta;[\eta,1])} \\
{} & \underline{\sf W}(Y_1,\ldots,Y_{n-1};T[SY_n,X]) \ar[rr]_-{\underline{\sf W}((1);[\eta,1] \circ \psi)}^-{\cong} && \underline{\sf W}(Y_1,\ldots,Y_{n-1};[Y_n,TX])
}
\end{equation*}
\normalsize
which commutes by the associativity of composition in $\underline{\sf W}$, and
in which the diagonal composite and the bottom morphism are isomorphisms by the induction hypothesis and the assumption (\ref{kellynm}) of the theorem respectively. Hence the morphism (\ref{kellyn}) is an isomorphism, and therefore by induction $\widehat{T}$ satisfies the hypothesis of Lemma \ref{adjfunlemma}.
\end{proof}

\begin{remark}
The condition of Theorem \ref{adjfunthm} is necessary, since if $S \dashv \widehat{T} \colon T_{\ast}\underline{\sf V} \lra \underline{\sf W}$ is an adjunction of ${\sf W}$-enriched multicategories, then it follows from the triangle identities and the ${\sf W}$-naturality of the unit $\eta$ and counit $\varepsilon$ of this adjunction that the composite morphism
\begin{equation*}
\cd[@C=3em]{
[Y,TX] \ar[r]^-S & T[SY,STX] \ar[r]^-{T[1,\varepsilon_X]} & T[SY,X]
}
\end{equation*}
is inverse to the morphism (\ref{kellynm}) in ${\sf W}$.
\end{remark}

Finally, to prove that a  symmetric multifunctor $T$ as in Theorem \ref{adjfunthm} is itself the right adjoint of an adjunction of symmetric multicategories, it is necessary and sufficient to further assume the $n=0$ case of the hypothesis of Lemma \ref{adjfunlemma}. This hypothesis is important in its own right, as it ensures that change of base along $T$ preserves underlying (multi)categories, and so we enshrine it in the following definition.

\begin{definition} \label{pronormaldef}
A symmetric multifunctor $T \colon {\sf V} \lra {\sf W}$ is said to be \emph{pronormal} if the function
\begin{equation} \label{pronormal}
\cd{
VX = {\sf V}(\,\,;X) \ar[r]^-T & {\sf W}(\,\,;TX) = WTX
}
\end{equation}
is a bijection for each object $X \in {\sf V}$.
\end{definition}

\begin{proposition} \label{pronormalprop}
The right adjoint of an adjunction of symmetric multicategories is pronormal.
\end{proposition}
\begin{proof}
Let $S \dashv T \colon {\sf V} \lra {\sf W}$ be an adjunction of symmetric multicategories. It follows from the triangle identity $T\varepsilon_X \circ \eta_{TX} = 1_{TX}$ and the multinaturality of the unit $\eta$ and counit $\varepsilon$ of this adjunction that the composite function
\begin{equation*}
\cd{
WTX \ar[r]^-S & VSTX \ar[r]^-{V\varepsilon_X} & VX
}
\end{equation*}
is inverse to the function (\ref{pronormal}) for each object $X \in{\sf V}$.
\end{proof}

\begin{remark}
Proposition \ref{pronormalprop} generalises the corresponding result for adjunctions of symmetric monoidal categories \cite[Proposition 2.1]{MR0360749}, wherein such a symmetric monoidal functor  is called ``normal'', however this conflicts with modern usage.
\end{remark}

\begin{theorem} \label{ordadjfunthm}
Let $T \colon {\sf V} \lra {\sf W}$ be a symmetric multifunctor between symmetric closed multicategories satisfying the hypothesis of Theorem {\normalfont\ref{adjfunthm}}. If $T$ is moreover pronormal, then $T$ is the right adjoint of an adjunction of symmetric multicategories.
\end{theorem}
\begin{proof}
Change of base along the underlying set symmetric multifunctor $W \colon {\sf W} \lra {\sf Set}$ sends the adjunction $S \dashv \widehat{T}$ of Theorem \ref{adjfunthm} to an adjunction of symmetric multicategories between $W_{\ast}T_{\ast}\underline{\sf V}$ and $W_{\ast}\underline{\sf W}$. If $T$ is pronormal, then the right adjoint of this adjunction is isomorphic to  $T \colon {\sf V} \lra {\sf W}$.
\end{proof}

\begin{remark}
The conditions of Theorem \ref{ordadjfunthm} are necessary, since if $S \dashv T \colon {\sf V} \lra {\sf W}$ is an adjunction of symmetric multicategories, then $T$ is pronormal by Proposition \ref{pronormalprop}, and it follows from the triangle identities  and the multinaturality of the unit $\eta$ and counit $\varepsilon$ of this adjunction  that the composite morphism
\begin{equation*}
\cd[@C=3em]{
[Y,TX] \ar[r]^-{\eta} & TS[Y,TX] \ar[r]^-{T\psi} & T[SY,STX] \ar[r]^-{T[1,\varepsilon_X]} & T[SY,X]
}
\end{equation*}
is inverse to the morphism (\ref{kellynm}) in ${\sf W}$.
\end{remark}

\section{The strictification adjunction} \label{stadjsect}
In this section we apply Theorem \ref{adjfunthm} to the inclusion of symmetric multicategories \linebreak ${\sf Gray} \lra {\sf Bicat}$ to prove that the strictification adjunction (\ref{stadj}) underlies an adjunction of ${\sf Bicat}$-enriched categories and, moreover, an adjunction of ${\sf Bicat}$-enriched symmetric multicategories. In fact, it will be useful for future applications (and no more difficult) to prove a stronger result, namely that the strictification adjunction underlies an adjunction of symmetric multicategories enriched over the symmetric multicategory ${\sf PsDbl}$ of \emph{pseudo double categories}, which we prove by applying  Theorem \ref{adjfunthm} to the inclusion ${\sf Dbl} \lra {\sf PsDbl}$ of the symmetric multicategory of double categories into ${\sf PsDbl}$. We introduce these symmetric multicategories of (pseudo) double categories in Appendix \ref{appendix}, and recall their pertinent details below.

  Strictification of bicategories was studied in \cite[\S 4.10]{MR1261589}, \cite[\S 2]{MR3076451}, and \cite{MR3023916}, and was generalised to pseudo double categories in \cite[\S 7]{MR1716779}. To begin, we recall the basic theory of this construction; we refer to the given references for further details. To avoid repetition, we speak in terms of pseudo double categories; the theory of strictification for bicategories is recovered by identifying bicategories with those pseudo double categories whose underlying categories of objects and vertical morphisms are discrete. (Note that this defines a fully faithful functor from the category of bicategories and pseudofunctors to the category of pseudo double categories and pseudo double functors, which has a right adjoint that sends a pseudo double category $A$ to its underlying bicategory $\mathbf{H}A$ of objects, horizontal morphisms, and globular cells.)
  
  Given a path of horizontal morphisms $f_1 \colon a_0 \lra a_1, \ldots, f_n\colon a_{n-1} \lra a_n$ in a pseudo double category $A$ (for $n\geq 0$),  we define a horizontal morphism $\varepsilon(f_1,\ldots,f_n) \colon a_0 \lra a_n$ in $A$ recursively as follows: for $n=0$, define $\varepsilon(\,)$ to be the horizontal identity morphism, for $n= 1$, define $\varepsilon(f) = f$, and for $n \geq 2$, define $\varepsilon(f_1,\ldots,f_n) = f_n \cdot \varepsilon(f_1,\ldots,f_{n-1})$.
  
  The \emph{strictification} of a pseudo double category $A$ is the (strict) double category $\st A$ with the same underlying category as $A$, whose horizontal morphisms $(f_1,\ldots,f_n) \colon a \lra b$ are paths of horizontal morphisms in $A$, with horizontal composition given by concatenation of paths, and whose cells $\alpha \colon (f_1,\ldots,f_n) \lra (g_1,\ldots,g_m)$ are given by cells $\alpha \colon \varepsilon(f_1,\ldots,f_n) \lra \varepsilon(g_1,\ldots,g_m)$ in $A$. Vertical composition of cells is as in $A$, and horizontal composition of  cells is defined using the horizontal composition and coherence isomorphisms of $A$. 
  
  For each horizontal morphism $(f_1,\ldots,f_n)$ in $\st A$, let $\kappa \colon (f_1,\ldots,f_n) \lra (\varepsilon(f_1,\ldots,f_n))$ denote the invertible globular cell in $\st A$ given by the identity cell on the horizontal morphism $\varepsilon(f_1,\ldots,f_n)$ in $A$. Note that every cell $(f_1,\ldots,f_n) \lra (g_1,\ldots,g_m)$ of $\st A$ is equal to the composite
 \begin{equation} \label{everycell}
 \cd{
 (f_1,\ldots,f_n) \ar[r]^-{\kappa} & (\varepsilon(f_1,\ldots,f_n)) \ar[r]^-{\alpha} & (\varepsilon(g_1,\ldots,g_m)) \ar[r]^-{\kappa^{-1}} & (g_1,\ldots,g_m)
 }
 \end{equation}
 for a unique cell $\alpha \colon \varepsilon(f_1,\ldots,f_n) \lra \varepsilon(g_1,\ldots,g_m)$ in $A$.
  
Let $\eta_A \colon A \lra \st A$ denote the pseudo double functor that is the identity on underlying categories and that sends a horizontal morphism $f$ to the unary path $(f)$ and a cell $\alpha \colon f \lra g$ to the cell $\alpha \colon (f) \lra (g)$ given by $\alpha$. The unit and composition constraints of the pseudo double functor $\eta_A$ are instances of the invertible globular cells $\kappa$.  
    
Let $\mathbf{Dbl}$ and $\mathbf{PsDbl}$ denote the category of double categories and double functors and the category of pseudo double categories and pseudo double functors respectively. The ``one-dimensional'' universal property of the strictification $\st A$ of a pseudo double category $A$ (the proof of which we recall as part of the proof of Theorem \ref{3duniprop} below) states that the composite function
\begin{equation*}
\cd{
\mathbf{Dbl}(\st A,B) \ar[r] & \mathbf{PsDbl}(\st A,B) \ar[rr]^-{\mathbf{PsDbl}(\eta_A,1)} && \mathbf{PsDbl}(A,B)
}
\end{equation*}
is a bijection for each double category $B$. It then follows by a standard categorical argument (which we generalised to enriched symmetric multicategories in Lemma \ref{adjfunlemma}) that there exists a unique extension of the above data to an adjunction
\begin{equation} \label{stadjdbl}
\cd{
\mathbf{Dbl} \ar@<-1.5ex>[rr]^-{\hdash}_-{} && \ar@<-1.5ex>[ll]_-{\st} \mathbf{PsDbl}
}
\end{equation}
whose left adjoint sends a pseudo double category to its strictification and whose unit has components given by the pseudo double functors $\eta_A \colon A \lra \st A$.

In Appendix \ref{appendix}, we introduce the symmetric multicategory ${\sf PsDbl}$ of pseudo double categories and its sub symmetric multicategory ${\sf Dbl}$ of double categories. Both are symmetric closed multicategories: for pseudo double categories $A$ and $B$, their internal hom in ${\sf PsDbl}$ is the pseudo double category $\underline{\mathbf{Hom}}(A,B)$ whose objects are pseudo double functors from $A$ to $B$, whose vertical morphisms are vertical transformations, whose horizontal morphisms are horizontal pseudo transformations, and whose cells are modifications; if $A$ and $B$ are double categories, their internal hom in ${\sf Dbl}$ is the full sub double category $\underline{\mathbf{Ps}}(A,B)$ of $\underline{\mathbf{Hom}}(A,B)$ on the (strict) double functors; note that if $B$ is a double category, then so is $\underline{\mathbf{Hom}}(A,B)$. In particular, if $A$ and $B$ are bicategories (seen as pseudo double categories with discrete underlying categories), then the vertical morphisms of $\underline{\mathbf{Hom}}(A,B)$ are icons, and its underlying bicategory is the usual hom bicategory $\mathbf{Hom}(A,B)$. 

We now prove the higher ``three-dimensional'' universal property of strictification, which shows that the inclusion of symmetric multicategories ${\sf Dbl} \lra {\sf PsDbl}$ satisfies the hypothesis of Theorem \ref{adjfunthm}.

\begin{theorem} \label{3duniprop}
For every pseudo double category $A$ and double category $B$, the composite double functor
\begin{equation} \label{dblfun}
\cd{
\underline{\Ps}(\st A,B) \ar[r] & \underline{\Hom}(\st A,B) \ar[rr]^-{\underline{\Hom}(\eta_A,1)} && \underline{\Hom}(A,B)
}
\end{equation}
is an isomorphism of double categories.
\end{theorem}
\begin{proof}
The ``one-dimensional'' universal property of strictification states that the double functor (\ref{dblfun}) is bijective on objects. To prove this, let $F \colon A \lra B$ be a pseudo double functor. Define $\overline{F} \colon \st A \lra B$ to be the (strict) double functor that agrees with $F$ on underlying categories, that sends a horizontal morphism $(f_1,\ldots,f_n)$ in $\st A$ to the horizontal morphism $\varepsilon(Ff_1,\ldots,Ff_n)$ in $B$, and that sends a cell in $\st A$ of the form (\ref{everycell}) to the following vertical composite cell in $B$,
\begin{equation*}
\cd{
\varepsilon(Ff_1,\ldots,Ff_n) \ar[r]^-{\varphi} & F\varepsilon(f_1,\ldots,f_n) \ar[r]^-{F\alpha} & F\varepsilon(g_1,\ldots,g_m) \ar[r]^-{\varphi^{-1}} & \varepsilon(Fg_1,\ldots,Fg_m)
}
\end{equation*}
where $\varphi \colon \varepsilon(Ff_1,\ldots,Ff_n) \lra F\varepsilon(f_1,\ldots,f_n)$ is the canonical invertible globular cell in $B$ defined recursively as follows: for $n=0$, define $\varphi$ to be the unit constraint $\varphi_0 \colon 1_{Fa} \lra F1_a$ of the pseudo double functor $F$, for $n=1$, define $\varphi$ to be the identity $1_{Ff} \colon Ff \lra Ff$, and for $n \geq 2$, define $\varphi$ to be the following vertical composite cell in $B$,
\begin{equation*}
\cd{
Ff_n \cdot \varepsilon(Ff_1,\ldots,Ff_{n-1}) \ar[r]^-{1 \cdot \varphi} & Ff_n \cdot F\varepsilon(f_1,\ldots,f_{n-1}) \ar[r]^-{\varphi_2} & F\varepsilon(f_1,\ldots,f_n)
}
\end{equation*}
where $\varphi_2$ denotes the composition constraint of the pseudo double functor $F$. Vertical functoriality of $\overline{F}$ follows from that of $F$, and horizontal functoriality of $\overline{F}$ is immediate for morphisms and is proved for cells by the coherence theorem for pseudofunctors applied to the underlying pseudofunctor $\mathbf{H}F \colon \mathbf{H}A \lra \mathbf{H}B$ of $F$ (see \cite[\S 2.3]{MR3076451}).

By the above definition, we see that $\overline{F}$ is the unique double functor satisfying the equation $\overline{F} \circ \eta_A = F$ of pseudo double functors. This is immediate on underlying categories, and holds for horizontal morphisms by the horizontal functoriality axioms of a double functor. To see that $\overline{F}$ is uniquely determined on cells by this equation and by the vertical functoriality axioms of a double functor, observe that the equation asserts in particular that $\overline{F}$ sends the pseudo double functor constraints of $\eta$ to the corresponding constraints of $F$, which are instances of the canonical invertible globular cells $\kappa$ and $\varphi$ respectively, and hence moreover implies that for each horizontal morphism $(f_1,\ldots,f_n)$ in $\st A$, the double functor $\overline{F}$ sends $\kappa \colon (f_1,\ldots,f_n) \lra (\varepsilon(f_1,\ldots,f_n))$ to $\varphi \colon \varepsilon(Ff_1,\ldots,Ff_n) \lra F\varepsilon(f_1,\ldots,f_n)$, as can  be proved by induction using the fact that for $n \geq 2$, $\kappa$ is equal to the following composite.
\begin{equation} \label{kappacomp}
\cd{
(f_{n}) \cdot (f_1,\ldots,f_{n-1}) \ar[r]^-{1 \cdot \kappa} & (f_{n}) \cdot \eta\varepsilon(f_1,\ldots,f_{n-1}) \ar[r]^-{\kappa} & \eta\varepsilon(f_1,\ldots,f_{n})
}
\end{equation}

It remains to prove that the double functor (\ref{dblfun}) is fully faithful on vertical morphisms, horizontal morphisms, and cells.
Let $F,G \colon \st A \lra B$ be double functors and let $\sigma \colon F\circ\eta_A \lra G\circ\eta_A$ be a vertical transformation. Define $\overline{\sigma} \colon F \lra G$ to be the vertical transformation whose component at an object $a \in A$ is the vertical morphism $\overline{\sigma}_a = \sigma_a \colon Fa \lra Ga$, and whose component at a horizontal morphism $(f_1,\ldots,f_n)$ in $\st A$ is the cell $\overline{\sigma}_{(f_1,\ldots,f_n)}$  defined recursively  as follows: for $n=0$, define $\overline{\sigma}_{() \colon a \to a}$ to be the identity cell on $\sigma_a$, and for $n \geq 1$, define $\overline{\sigma}_{(f_1,\ldots,f_{n})}$ to be the following composite in $B$.
\begin{equation*}
\cd{
Fa \ar[rr]^-{F(f_1,\ldots,f_{n-1})} \ar[d]_-{\sigma_a} \dtwocell{dr}{\overline{\sigma}_{(f_1,\ldots,f_{n-1})}} &{} & Fb \ar[rr]^-{F(f_{n})} \ar[d]^-{\sigma_b} \dtwocell{drr}{\sigma_{f_{n}}} && Fc \ar[d]^-{\sigma_c} \\
Ga \ar[rr]_-{G(f_1,\ldots,f_{n-1})} &{} & Gb \ar[rr]_-{G(f_{n})} && Gc
}
\end{equation*}
The horizontal functoriality of $\overline{\sigma}$ follows by a standard induction from the definition of $\overline{\sigma}$ and the functoriality of $\sigma$. The naturality of $\overline{\sigma}$ with respect to vertical morphisms is immediate from the corresponding property for $\sigma$. To prove the naturality of $\overline{\sigma}$ with respect to cells, we use the following result. 

We prove by induction on $n \geq 0$ that for each horizontal morphism $(f_1,\ldots,f_n)$ in $\st A$, the following equation holds in $B$.
\begin{equation*}
\cd[@C=2.5em]{
Fa \ar[rr]^-{F(f_1,\ldots,f_n)} \ar@{=}[d] \dtwocell{drr}{F\kappa} & {} & Fb \ar@{=}[d] \\
Fa \ar[rr]|-{F\eta\varepsilon(f_1,\ldots,f_n)} \ar[d]_-{\sigma_a} \dtwocell{dr}{\sigma_{\varepsilon(f_1,\ldots,f_n)}} & {} & Fb \ar[d]^-{\sigma_b} \\
Ga \ar[rr]_-{G\eta\varepsilon(f_1,\ldots,f_n)} & {} & Gb
}
\quad
=
\quad
\cd[@C=2.5em]{
Fa \ar[rr]^-{F(f_1,\ldots,f_n)} \ar[d]_-{\sigma_a} \dtwocell{dr}{\overline{\sigma}_{(f_1,\ldots,f_n)}} & {} & Fb \ar[d]^-{\sigma_b} \\
Ga \ar[rr]|-{G(f_1,\ldots,f_n)} \ar@{=}[d] \dtwocell{drr}{G\kappa} & {} & Gb \ar@{=}[d] \\
Ga \ar[rr]_-{G\eta\varepsilon(f_1,\ldots,f_n)} & {} & Gb
}
\end{equation*}
For $n=0$, this equation is the horizontal unit axiom for the vertical transformation $\sigma$, since $\kappa$ in this case is the unit constraint for the pseudo double functor $\eta_A \colon A \lra \st A$. For $n = 1$, the equation is immediate. For $n \geq 2$, the equation follows from the following commutative diagram in the category $B_1$ (i.e.\ the category of horizontal morphisms and cells in $B$),
\begin{equation*}
\cd[@C=2.5em]{
F(f_{n}) \cdot F(f_1,\ldots,f_{n-1}) \ar[r]^-{1\cdot F\kappa} \ar[d]_-{\sigma_{f_{n}}\cdot 1} & F(f_{n}) \cdot F\eta\varepsilon(f_1,\ldots,f_{n-1}) \ar[r]^-{F\kappa} \ar[d]^-{\sigma_{f_{n}} \cdot 1} & F\eta\varepsilon(f_1,\ldots,f_{n}) \ar[dd]^-{\sigma_{\varepsilon(f_1,\ldots,f_{n})}}\\
G(f_{n}) \cdot F(f_1,\ldots,f_{n-1}) \ar[r]^-{1\cdot F\kappa} \ar[d]_-{1 \cdot \overline{\sigma}_{(f_1,\ldots,f_{n-1})}} & G(f_{n}) \cdot F\eta\varepsilon(f_1,\ldots,f_{n-1}) \ar[d]^-{1 \cdot \sigma_{\varepsilon(f_1,\ldots,f_{n-1})}}\\
G(f_{n}) \cdot G(f_1,\ldots,f_{n-1}) \ar[r]_-{1 \cdot G\kappa} & G(f_{n}) \cdot G\eta \varepsilon (f_1,\ldots,f_{n-1}) \ar[r]_-{G\kappa} & G\eta \varepsilon (f_1,\ldots,f_{n})
}
\end{equation*}
whose right-hand region commutes by the horizontal functoriality of the vertical transformation $\sigma$, and whose lower left-hand region commutes  by the induction hypothesis. Here we have used that $\kappa$ for $n=2$ is the composition constraint for the pseudo double functor $\eta_A$, and that for $n \geq 2$ the  isomorphism $\kappa_{(f_1,\ldots,f_{n})}$ is equal to the composite (\ref{kappacomp}).

Since any cell $(f_1,\ldots,f_n) \lra (g_1,\ldots,g_m)$ in $\st A$ is equal to one of the form (\ref{everycell}) 
for some cell $\alpha \colon \varepsilon(f_1,\ldots,f_n) \lra \varepsilon(g_1,\ldots,g_m)$ in $A$, naturality of $\overline{\sigma}$ with respect to cells follows from the following commutative diagram,
\begin{equation*}
\cd{
F(f_1,\ldots,f_n) \ar[r]^-{F\kappa} \ar[d]_-{\overline{\sigma}_{(f_1,\ldots,f_n)}} & F\eta\varepsilon(f_1,\ldots,f_n) \ar[r]^-{F\eta\alpha} \ar[d]_-{\sigma_{\varepsilon(f_1,\ldots,f_n)}} & F\eta\varepsilon(g_1,\ldots,g_m) \ar[r]^-{F\kappa^{-1}} \ar[d]^-{\sigma_{\varepsilon(g_1,\ldots,g_m)}}& F(g_1,\ldots,g_m) \ar[d]^-{\overline{\sigma}_{(g_1,\ldots,g_m)}} \\
G(f_1,\ldots,f_n) \ar[r]_-{G\kappa} & G\eta\varepsilon(f_1,\ldots,f_n) \ar[r]_-{G\eta\alpha} & G\eta\varepsilon(g_1,\ldots,g_m) \ar[r]_-{G\kappa^{-1}}& G(g_1,\ldots,g_m) 
}
\end{equation*}
whose centre region commutes by cell naturality of $\sigma$, and whose outer regions commute by the result of the previous paragraph. 

By the above definition and by the horizontal functoriality axioms for a vertical transformation, we see that  $\overline{\sigma}$ is the unique vertical transformation satisfying the equation $\overline{\sigma} \circ \eta_A = \sigma$. Hence the double functor (\ref{dblfun}) is fully faithful on vertical morphisms.

The proof that the double functor (\ref{dblfun}) is fully faithful on horizontal morphisms essentially follows the same argument as for vertical morphisms. Let $F,G \colon \st A \lra B$ be double functors, and let $\theta \colon F\circ\eta_A \lra G\circ\eta_A$ be a pseudo horizontal transformation. Define $\overline{\theta} \colon F \lra G$ to be the pseudo horizontal transformation whose component at an object $a \in A$ is the horizontal morphism $\overline{\theta}_a = \theta_a$, whose component at a  vertical morphism $u \colon a \lra b$ of $A$ is the cell $\overline{\theta}_u = \theta_u$, and whose component at a horizontal morphism $(f_1,\ldots,f_n)$ in $\st A$ is the invertible globular cell $\overline{\theta}_{(f_1,\ldots,f_n)}$ defined recursively as follows: for $n=0$, define $\overline{\theta}_{() \colon a \to a}$ to be the identity cell on $\theta_a$, and for $n \geq 1$, define $\overline{\theta}_{(f_1,\ldots,f_{n})}$ to be the following pasting composite of globular cells in $B$.
\begin{equation*}
\cd{
Fa \ar[rr]^-{F(f_1,\ldots,f_{n-1})} \ar[d]_-{\theta_a} \dtwocell{dr}{\overline{\theta}_{(f_1,\ldots,f_{n-1})}} &{} & Fb \ar[rr]^-{F(f_{n})} \ar[d]^-{\theta_b} \dtwocell{drr}{\theta_{f_{n}}} && Fc \ar[d]^-{\theta_c} \\
Ga \ar[rr]_-{G(f_1,\ldots,f_{n-1})} &{} & Gb \ar[rr]_-{G(f_{n})} && Gc
}
\end{equation*}
The horizontal functoriality of $\overline{\theta}$ follows by a standard induction from the definition of $\overline{\theta}$ and the functoriality of $\theta$. The vertical functoriality of $\overline{\theta}$ is immediate from the corresponding property for $\theta$. To prove the naturality of $\overline{\sigma}$ with respect to cells, we use the following result. 

We prove by induction on $n \geq 0$ that for each horizontal morphism $(f_1,\ldots,f_n)$ in $\st A$, the following square (whose arrows are invertible globular cells in $B$) commutes.
\begin{equation*}
\cd{
\theta_b \cdot F(f_1,\ldots,f_n) \ar[r]^-{1 \cdot F\kappa} \ar[d]_-{\overline{\theta}_{(f_1,\ldots,f_n)}} & \theta_b \cdot F\eta\varepsilon(f_1,\ldots,f_n) \ar[d]^-{\theta_{\varepsilon(f_1,\ldots,f_n)}} \\
G(f_1,\ldots,f_n) \cdot \theta_a \ar[r]_-{G\kappa \cdot 1} & G\eta\varepsilon(f_1,\ldots,f_n) \cdot \theta_a
}
\end{equation*}
For $n=0$, this equation is the horizontal unit axiom 
for the pseudo horizontal transformation $\theta \colon F\circ\eta_A \lra G\circ\eta_A$. For $n = 1$, the equation is immediate. For $n \geq 2$, the equation follows from the commutative diagram
\begin{equation*}
\cd[@C=2.5em]{
\theta_c \cdot F(f_{n}) \cdot F(f_1,\ldots,f_{n-1}) \ar[r]^-{1\cdot1\cdot F\kappa} \ar[d]_-{\theta_{f_{n}}\cdot 1} & \theta_c \cdot F(f_{n}) \cdot F\eta\varepsilon(f_1,\ldots,f_{n-1}) \ar[r]^-{1 \cdot F\kappa} \ar[d]^-{\theta_{f_{n}} \cdot 1} & \theta_c \cdot F\eta\varepsilon(f_1,\ldots,f_{n}) \ar[dd]^-{\theta_{\varepsilon(f_1,\ldots,f_{n})}}\\
G(f_{n}) \cdot \theta_b \cdot F(f_1,\ldots,f_{n-1}) \ar[r]^-{1 \cdot 1\cdot F\kappa} \ar[d]_-{1 \cdot \overline{\theta}_{(f_1,\ldots,f_{n-1})}} & G(f_{n}) \cdot \theta_b \cdot F\eta\varepsilon(f_1,\ldots,f_{n-1}) \ar[d]^-{1 \cdot \theta_{\varepsilon(f_1,\ldots,f_{n-1})}}\\
G(f_{n}) \cdot G(f_1,\ldots,f_{n-1}) \cdot \theta_a \ar[r]_-{1 \cdot G\kappa \cdot 1} & G(f_{n}) \cdot G\eta \varepsilon (f_1,\ldots,f_{n-1}) \cdot \theta_a \ar[r]_-{G\kappa \cdot 1} & G\eta \varepsilon (f_1,\ldots,f_{n}) \cdot \theta_a
}
\end{equation*}
whose right-hand region commutes by the horizontal functoriality of the pseudo horizontal transformation $\theta$, and whose lower left-hand region commutes by the induction hypothesis. 

Since any cell $(f_1,\ldots,f_n) \lra (g_1,\ldots,g_m)$ in $\st A$ is equal to one of the form (\ref{everycell}) for some cell $\alpha \colon \varepsilon(f_1,\ldots,f_n) \lra \varepsilon(g_1,\ldots,g_m)$ in $A$ (with source and target vertical morphisms $u$ and $v$, say), naturality of $\overline{\theta}$ with respect to cells follows from the following commutative diagram in the category $B_1$,
\begin{equation*}
\cd[@C=2.43em]{
\theta_b \cdot F(f_1,\ldots,f_n) \ar[r]^-{1 \cdot F\kappa} \ar[d]_-{\overline{\theta}_{(f_1,\ldots,f_n)}} & \theta_b \cdot F\eta\varepsilon(f_1,\ldots,f_n) \ar[r]^-{\theta_v \cdot F\eta\alpha} \ar[d]_-{\theta_{\varepsilon(f_1,\ldots,f_n)}} & \theta_d \cdot F\eta\varepsilon(g_1,\ldots,g_m) \ar[r]^-{1 \cdot F\kappa^{-1}} \ar[d]^-{\theta_{\varepsilon(g_1,\ldots,g_m)}}& \theta_d \cdot F(g_1,\ldots,g_m) \ar[d]^-{\overline{\theta}_{(g_1,\ldots,g_m)}} \\
G(f_1,\ldots,f_n) \cdot \theta_a \ar[r]_-{G\kappa \cdot 1} & G\eta\varepsilon(f_1,\ldots,f_n) \cdot \theta_a \ar[r]_-{G\eta\alpha \cdot \theta_u} & G\eta\varepsilon(g_1,\ldots,g_m)\cdot \theta_a \ar[r]_-{G\kappa^{-1} \cdot 1}& G(g_1,\ldots,g_m) \cdot \theta_a
}
\end{equation*}
whose centre region commutes by cell naturality of $\theta$, and whose outer regions commute by the result of the previous paragraph. 

By the above definition and by the horizontal functoriality axioms for a pseudo horizontal transformation, we see that  $\overline{\theta}$ is the unique pseudo horizontal transformation satisfying the equation $\overline{\theta} \circ \eta_A = \theta$. Hence the double functor (\ref{dblfun}) is fully faithful on horizontal morphisms.

It remains to show that the double functor (\ref{dblfun}) is fully faithful on cells. Let $\theta \colon F \lra G$ and $\varphi \colon H \lra K$ be pseudo horizontal transformations and let $\sigma \colon F \lra H$ and $\tau \colon G \lra K$ be vertical transformations between double functors $\st A\lra B$. Let $m$ be a modification as on the left below.
\begin{equation*}
\cd{
F\circ\eta_A \ar[r]^-{\theta\circ\eta_A} \ar[d]_-{\sigma\circ\eta_A} \dtwocell{dr}{m} & G\circ\eta_A \ar[d]^-{\tau\circ\eta_A} \\
H\circ\eta_A \ar[r]_-{\varphi\circ\eta_A} & K\circ \eta_A
}
\qquad
\qquad
\cd{
F \ar[r]^-{\theta} \ar[d]_-{\sigma} \dtwocell{dr}{\overline{m}} & G\ar[d]^-{\tau} \\
H \ar[r]_-{\varphi} & K 
}
\end{equation*}
Define $\overline{m}$ to be the modification as on the right above whose component at an object $a \in A$ is the cell $\overline{m}_a = m_a$. Vertical naturality of $\overline{m}$ follows immediately from that of $m$, and horizontal naturality is proved by a standard induction. Moreover, $\overline{m}$ is evidently the unique modification satisfying the equation $\overline{m} \circ \eta_A = m$. Hence the double functor (\ref{dblfun}) is fully faithful on cells, and is therefore an isomorphism of double categories.
\end{proof}

We can deduce from this theorem that the  inclusion of symmetric multicategories \linebreak ${\sf Gray} \lra {\sf Bicat}$ also satisfies the hypothesis of Theorem \ref{adjfunthm}. Recall that for bicategories $A$ and $B$, their internal hom  in ${\sf Bicat}$ is the bicategory $\mathbf{Hom}(A,B)$ of pseudofunctors, pseudonatural transformations, and modifications, which is a $2$-category if $B$ is strict, and that if $A$ and $B$ are $2$-categories, their internal hom in ${\sf Gray}$ is the full sub-$2$-category $\mathbf{Ps}(A,B)$ of $\mathbf{Hom}(A,B)$ on the $2$-functors.

\begin{corollary} \label{3duniprop2}
For every bicategory $A$ and $2$-category $B$, the composite $2$-functor
\begin{equation*}
\cd{
\Ps(\st A,B) \ar[r] & \Hom(\st A,B) \ar[rr]^-{\Hom(\eta_A,1)} && \Hom(A,B)
}
\end{equation*}
is an isomorphism of $2$-categories.
\end{corollary}
\begin{proof}
Applying the underlying bicategory functor $\mathbf{H} \colon \mathbf{PsDbl} \lra \Bicat$ to the composite double functor of Theorem \ref{3duniprop} yields the stated result.
\end{proof}

Therefore the inclusions of symmetric multicategories ${\sf Dbl} \lra {\sf PsDbl}$ and ${\sf Gray} \lra {\sf Bicat}$ both satisfy the hypothesis of Theorem \ref{adjfunthm},  and we can deduce the following results.

\begin{theorem}
The strictification adjunction for pseudo double categories {\normalfont (\ref{stadjdbl})} underlies:
\begin{enumerate}[font=\normalfont, label=(\roman*)]
\item an adjunction of symmetric ${\sf PsDbl}$-multicategories,
\item an adjunction of ${\sf PsDbl}$-categories, and
\item an adjunction of symmetric multicategories.
\end{enumerate}
\end{theorem}
\begin{proof}
It follows from the three-dimensional universal property of strictification proved in Theorem \ref{3duniprop} that the inclusion  of symmetric multicategories ${\sf Dbl} \lra {\sf PsDbl}$ satisfies the hypothesis of Theorem \ref{adjfunthm}. Hence the induced symmetric ${\sf PsDbl}$-multifunctor, which is the inclusion of the symmetric ${\sf PsDbl}$-multicategory of double categories (i.e.\ the change of base along the inclusion ${\sf Dbl} \lra {\sf PsDbl}$ of the canonical self-enrichment of ${\sf Dbl}$) into the symmetric ${\sf PsDbl}$-multicategory of pseudo double categories (i.e.\ the canonical self-enrichment of ${\sf PsDbl}$), is the right adjoint of an adjunction (i) of symmetric ${\sf PsDbl}$-multicategories. 
The adjunction (ii) between the ${\sf PsDbl}$-categories of double categories and pseudo double categories can be obtained from the adjunction (i) by application of the $2$-functor ${\sf PsDbl}\text{-}\mathbf{SMult} \lra {\sf PsDbl}\text{-}\Cat$ that sends a  symmetric ${\sf PsDbl}$-multicategory to its underlying ${\sf PsDbl}$-category of unary morphisms.
 Since the inclusion of symmetric multicategories ${\sf Dbl} \lra {\sf PsDbl}$ is pronormal (in fact the functions (\ref{pronormal}) are identities), we have moreover that this inclusion satisfies the hypotheses of Theorem \ref{ordadjfunthm}, and hence is the right adjoint of an adjunction (iii) of symmetric multicategories.

To show that the strictification adjunction for pseudo double categories (\ref{stadjdbl}) underlies each of these adjunctions, it suffices to show that the inclusion of categories $\dbl \lra \psdbl$ is the underlying functor of their right adjoints. In each case, this is an immediate consequence of the fact that change of base along the inclusion of symmetric multicategories ${\sf Dbl} \lra {\sf PsDbl}$ is pronormal and hence preserves underlying categories.
\end{proof}

\begin{theorem} \label{fincor}
The strictification adjunction for bicategories {\normalfont(\ref{stadj})} underlies:
\begin{enumerate}[font=\normalfont, label=(\roman*)]
\item an adjunction of symmetric ${\sf PsDbl}$-multicategories,
\item an adjunction of ${\sf PsDbl}$-categories,
\item an adjunction of symmetric ${\sf Bicat}$-multicategories,
\item an adjunction of ${\sf Bicat}$-categories, and 
\item an adjunction of symmetric multicategories.
\end{enumerate}
\end{theorem}
\begin{proof}
Since the strictification of a bicategory is a $2$-category (when seen as (pseudo) double categories with discrete underlying categories), the adjunctions (i), (ii), and (iii) of the previous theorem restrict to the adjunctions (i), (ii), and (v) of the present theorem. The adjunctions (iii) and (iv) can be obtained from the adjunctions (i) and (ii) by change of base along the symmetric multifunctor $\mathbf{H} \colon {\sf PsDbl} \lra {\sf Bicat}$ of (\ref{appadj}) that sends a pseudo double category to its underlying bicategory, or alternatively from Theorem \ref{adjfunthm} and Corollary \ref{3duniprop2} by the argument of the previous theorem applied to the inclusion of symmetric multicategories ${\sf Gray} \lra {\sf Bicat}$. It again follows from the pronormality of this inclusion that
the strictification adjunction for bicategories (\ref{stadj}) underlies each of the adjunctions (i)--(v) in the statement.
\end{proof}

\begin{remark}
We recover the strictification $2$-adjunction (between the $2$-categories $\twocat_2$ and $\Bicat_2$ whose $2$-cells are icons) of \cite[Theorem 4.1]{MR3023916} as the change of base of the ${\sf PsDbl}$-enriched strictification adjunction of Theorem \ref{fincor}(ii) along the symmetric multifunctor $\mathbf{U} \colon {\sf PsDbl} \lra {\sf Cat}$ that sends a pseudo double category to its underlying category.
\end{remark}

\begin{remark}
Compare Theorem \ref{fincor}(iii) with \cite[Remark 5.8]{MR3023916}, where it is stated that one could show that the strictification functor is a symmetric monoidal functor (in some tricategorical sense) between symmetric monoidal tricategories. To actually show this (and indeed to make the necessary definitions)  would be no mean feat, whereas we have captured the same three-dimensional structure of the strictification functor within the framework of ordinary enriched category theory, namely as a symmetric multifunctor between symmetric multicategories enriched over the symmetric multicategory ${\sf Bicat}$.
\end{remark}

We end this section with a couple of applications of Theorem \ref{fincor}. Change of base along the adjunction of symmetric multicategories of Theorem \ref{fincor}(v) induces an adjunction between the categories of $\Gray$-categories and ${\sf Bicat}$-categories,
\begin{equation*}
\cd{
\mathbf{Gray}\text{-}\Cat \ar@<-1.5ex>[rr]^-{\hdash}_-{} && \ar@<-1.5ex>[ll]_-{\st_{\ast}} {\sf Bicat}\text{-}\Cat
}
\end{equation*}
 whose left adjoint sends a ${\sf Bicat}$-category $\mathcal{A}$ to its change of base $\Gray$-category $\st_{\ast}\mathcal{A}$ along the symmetric multifunctor $\st \colon {\sf Bicat} \lra {\sf Gray}$ (recall that ${\sf Gray}$ is represented as a symmetric multicategory by the symmetric monoidal category $\Gray$). The component of the unit of this adjunction at a ${\sf Bicat}$-category $\mathcal{A}$ is the ${\sf Bicat}$-functor $\eta_{\ast} \colon \mathcal{A} \lra \st_{\ast}\mathcal{A}$, which is the identity on objects and is given on homs by the identity-on-objects biequivalences \linebreak $\eta \colon \mathcal{A}(A,B) \lra \st\,\mathcal{A}(A,B)$, and is therefore a triequivalence. In particular, taking $\mathcal{A}$ to be the self-enrichment of the symmetric closed multicategory ${\sf Bicat}$, which we denote by $\Bicat_3$, yields the following result.
 
\begin{proposition} \label{newgray}
The category of bicategories and pseudofunctors underlies a $\Gray$-category $\st_{\ast}\Bicat_3$ with hom $2$-categories $(\st_{\ast}\Bicat_3)(A,B) = \st\,\Hom(A,B)$, which is triequivalent to the ${\sf Bicat}$-category $\Bicat_3$ of bicategories via an identity-on-underlying-categories ${\sf Bicat}$-functor $\eta_{\ast} \colon \Bicat_3 \lra \st_{\ast}\Bicat_3$.
\end{proposition}

Explicitly,  $\st_{\ast}\Bicat_3$ is the $\Gray$-category whose objects are bicategories, whose morphisms are pseudofunctors, whose $2$-cells are ``vertical paths'' of pseudonatural transformations, and whose $3$-cells are modifications between the vertical composites of such paths. For bicategories $A$, $B$, and $C$, the composition $2$-functor out of the Gray tensor product
\begin{equation*}
\st\,\Hom(B,C) \otimes \st\,\Hom(A,B) \lra \st\,\Hom(A,C)
\end{equation*}
is the image of the composition two-variable pseudofunctor 
\begin{equation*}
(\Hom(B,C),\Hom(A,B)) \lra \Hom(A,C)
\end{equation*}
under the symmetric multifunctor $\st \colon {\sf Bicat} \lra {\sf Gray}$. Hence the interchange constraint (\ref{cinqgray})  for a ``horizontally composable'' pair of $2$-cells  
\begin{equation*}
(\alpha_i \colon f_{i-1} \lra f_i \colon A \lra B)_{1\leq i\leq n} \quad  \quad (\beta_j \colon g_{j-1} \lra g_j \colon B \lra C)_{1\leq j\leq m}
\end{equation*}
in the $\Gray$-category $\st_{\ast}\Bicat_3$ is the invertible $2$-cell 
 \begin{equation*}
 (g_0\alpha_1,\ldots,g_0\alpha_n,\beta_1f_n,\ldots,\beta_mf_n) \lra (\beta_1f_0,\ldots,\beta_mf_0,g_m\alpha_1,\ldots,g_m\alpha_n)
 \end{equation*}
  in the $2$-category $\st\,\Hom(A,C)$ represented by the following pasting composite in the bicategory $\Hom(A,C)$, 
\begin{equation*}
\cd[@=3.5em]{
g_0f_0 \dtwocell{dr}{\beta_1\alpha_1} \ar[d]_-{\beta_1f_0} \ar[r]^-{g_0\alpha_1} & g_0f_1  \ar@{..}[r] \ar[d]^-{\beta_1f_1} & g_0f_{n-1} \dtwocell{dr}{\beta_1\alpha_n} \ar[d]_-{\beta_1f_{n-1}} \ar[r]^-{g_0\alpha_n} & g_0f_n \ar[d]^-{\beta_1f_n} \\
g_1f_0  \ar@{..}[d] \ar[r]_-{g_1\alpha_1} & g_1f_1 \ar@{..}[dr] \ar@{..}[r]  \ar@{..}[d] & g_1f_{n-1} \ar[r]_-{g_1\alpha_n}  \ar@{..}[d]  & g_1f_n  \ar@{..}[d] \\
g_{m-1}f_0  \dtwocell{dr}{\beta_m\alpha_1} \ar[d]_-{\beta_mf_0} \ar[r]^-{g_{m-1}\alpha_1} & g_{m-1}f_1 \ar@{..}[r] \ar[d]^-{\beta_mf_1} & g_{m-1}f_{n-1} \dtwocell{dr}{\beta_m\alpha_n} \ar[d]_-{\beta_mf_{n-1}} \ar[r]^-{g_{m-1}\alpha_n} & g_{m-1}f_n \ar[d]^-{\beta_mf_n} \\
g_{m}f_0 \ar[r]_-{g_{m}\alpha_1} & g_{m}f_1 \ar@{..}[r] & g_{m}f_{n-1} \ar[r]_-{g_{m}\alpha_n} & g_{m}f_n \\
}
\end{equation*}
where the $\beta_j\alpha_i$ denote the interchange constraints (\ref{canmod}) of the ${\sf Bicat}$-category $\Bicat_3$. 
The identity-on-underlying-categories triequivalence ${\sf Bicat}$-functor $\eta_{\ast} \colon \Bicat_3 \lra \st_{\ast}\Bicat_3$ sends a pseudonatural transformation to the unary vertical path it comprises.
 Working instead in  the setting of tricategory theory, this defines an identity-on-objects, identity-on-morphisms triequivalence trihomomorphism from the tricategory of bicategories to the $\Gray$-category $\st_{\ast}\Bicat_3$ of bicategories. 

Finally, recall that the strictification  trihomomorphism from the tricategory of bicategories to the $\Gray$-category of $2$-categories restricts to a triequivalence between the tricategory of bicategories and the full sub-$\Gray$-category of the $\Gray$-category of $2$-categories on the cofibrant $2$-categories (see \cite[Theorem 8.21]{MR3076451}). Using the ${\sf PsDbl}$-enriched adjunction  of Theorem \ref{fincor}(ii), we can promote this to a \emph{biequivalence} of ${\sf PsDbl}$-categories, i.e.\ a ${\sf PsDbl}$-functor that is surjective on objects up to equivalence and is an equivalence on the hom pseudo double categories;  by definition, a morphism in a ${\sf PsDbl}$-category is an equivalence if it is an equivalence in the underlying $2$-category.  Recall that a $2$-category (and more generally a double category) is said to be \emph{cofibrant} if its (horizontal) underlying category is free on a graph \cite{MR1931220,MR2449004}. Let us denote the ${\sf Dbl}$-category of $2$-categories and the ${\sf PsDbl}$-category of bicategories by $\underline{\Gray_3}$ and $\underline{\Bicat_3}$ respectively.

\begin{proposition}
The ${\sf PsDbl}$-enriched strictification adjunction of Theorem {\normalfont \ref{fincor}(ii)}
\begin{equation*}
\cd{
\underline{\Gray_3} \ar@<-1.5ex>[rr]^-{\hdash}_-{} && \ar@<-1.5ex>[ll]_-{\st} \underline{\Bicat_3}
}
\end{equation*}
restricts to an adjoint biequivalence between the ${\sf PsDbl}$-category $\underline{\Bicat_3}$ of bicategories and the full sub-${\sf Dbl}$-category of $\underline{\Gray_3}$ on the cofibrant $2$-categories. 
\end{proposition}
\begin{proof}
Each component $\eta_A \colon A \lra \st A$ of the unit of this adjunction is a bijective-on-objects biequivalence, and is therefore an equivalence in the ${\sf PsDbl}$-category of bicategories. Furthermore, the strictification of a bicategory $A$ is a cofibrant $2$-category (since its underlying category is free on the underlying graph of $A$), and the component of the counit of this adjunction at a cofibrant $2$-category is a bijective-on-objects biequivalence with cofibrant domain and codomain, and is therefore an equivalence in the ${\sf Dbl}$-category of $2$-categories.
\end{proof}

Moreover, the same argument  shows that the ${\sf PsDbl}$-category of pseudo double categories is biequivalent to the full sub-${\sf Dbl}$-category of the ${\sf Dbl}$-category of double categories on the cofibrant double categories.

\appendix

\section{The multicategory of pseudo double categories} \label{appendix}
This appendix introduces the symmetric closed multicategory $\sf{PsDbl}$ of pseudo double categories, which generalises Verity's symmetric closed multicategory of bicategories \cite[\S 1.3]{MR2844536}.
We refer to the papers \cite{MR1716779} and \cite{MR2303000} (and the references contained therein) for the basic theory of pseudo double categories. Recall that every bicategory gives rise to a pseudo double category whose underlying category is discrete, with which we identify the bicategory. Note that we follow the convention that the ``weak'' direction in a pseudo double category is the horizontal,  so that the underlying category $A_0$ of a pseudo double category $A$ consists of its objects and vertical morphisms. 
We do not distinguish notationally between the arrows denoting the horizontal and vertical morphisms of a pseudo double category, though we will have occasion to denote the vertical and horizontal identities of an object $a$ by $1_a^v$ and $1_a^h$ respectively. 

We could proceed directly to define and prove the symmetric closed multicategory structure and axioms for $\sf{PsDbl}$, however this would be a rather lengthy affair. In fact it is possible to give a more streamlined definition and proof, for there is a significant amount of redundancy in the definition of a symmetric closed multicategory: for example, the multimorphisms are completely determined by the nullary morphisms and the internal hom objects. Inspired by \cite[\S4.2]{Bourke2015}, we will define $\sf{PsDbl}$ to be the symmetric closed multicategory of weak morphisms  of a certain symmetric skew closed structure on the category $\psdbl_{\mathrm{s}}$ of pseudo double categories and strict double functors.

A symmetric skew closed structure on a category $\mathcal{C}$ consists of an internal hom functor $[-,-] \colon \mathcal{C}^\mathrm{op} \times \mathcal{C} \lra \mathcal{C}$, a unit object $I$, natural transformations $L \colon [Y,Z] \lra [[X,Y],[X,Z]]$, $i \colon [I,X] \lra X$, and $j \colon I \lra [X,X]$, and a symmetry natural isomorphism \linebreak $s \colon [X,[Y,Z]] \cong [Y,[X,Z]]$, subject to axioms (see \cite[\S2]{MR3010098} and \cite[Definition 3.1]{bourkelackbraid}).
Associated to any symmetric skew closed category $\mathcal{C}$ is a symmetric closed multicategory ${\sf C}$, which we call its \emph{symmetric closed multicategory of weak morphisms}, whose objects and  internal hom objects are the same as those of $\mathcal{C}$, and whose sets of multimorphisms ${\sf C}(X_1,\ldots,X_n;Y)$ are defined recursively as follows: for $n=0$ define ${\sf C}(\,\,;Y) = \mathcal{C}(I,Y)$, and for $n \geq 1$ define ${\sf C}(X_1,\ldots,X_n;Y) = {\sf C}(X_1,\ldots,X_{n-1};[X_n,Y])$. The remainder of the symmetric closed multicategory structure and axioms follow from those of the symmetric skew closed category in a straightforward manner. Moreover, this defines a $2$-functor from the $2$-category of symmetric skew closed categories, symmetric closed functors, and closed natural transformations to the $2$-category of symmetric multicategories, symmetric multifunctors, and multinatural transformations. (See \cite{bourkelackskewmult} for an abstract proof of the non-symmetric version of this statement; see also \cite{MR2909641} for a more concrete proof of the non-symmetric non-skew version).

For each pair of pseudo double categories $A$ and $B$, let $\underline{\Hom}(A,B)$ denote the pseudo double category whose objects are pseudo double functors $A \lra B$, whose vertical morphisms are vertical transformations, whose horizontal morphisms are pseudo horizontal transformations, and whose cells are modifications. Presently, we shall show that these hom pseudo double categories are the internal hom objects of a symmetric skew closed structure on $\psdbl_{\mathrm{s}}$ whose unit object $I$ represents the functor $\mathrm{ob} \colon \psdbl_{\mathrm{s}} \lra \Set$, and hence also the internal hom objects of a symmetric closed multicategory $\sf{PsDbl}$ of pseudo double categories whose sets of binary morphisms are defined as  ${\sf PsDbl}(A,B;C) = \mathrm{ob}\,\underline{\Hom}(A,\underline{\Hom}(B,C))$. It will be convenient to have an explicit description of these binary morphisms, and of the pseudo double categories $\underline{\Hom}(A,\underline{\Hom}(B,C))$ they form.

Note that if $A$ and $B$ are bicategories, then $\underline{\Hom}(A,B)$ is the pseudo double category whose objects are pseudofunctors $A \lra B$, whose vertical morphisms are icons, whose horizontal morphisms are pseudonatural transformations, and whose cells are ``modification squares''.

\begin{definition} \label{deffuntwovar}
A \emph{pseudo double functor of two variables} $F \colon (A,B) \lra C$ consists of the following data:
\begin{enumerate}[(i)]
\item a functor $F \colon A_0 \times B_0 \lra C_0$,
\item for each object $a \in A$, a pseudo double functor $F(a,-) \colon B \lra C$ agreeing with the functor (i) on underlying categories, 
\item for each object $b \in B$, a pseudo double functor $F(-,b) \colon A \lra C$ agreeing with the functor (i) on underlying categories, 
\item for each vertical morphism $u \colon a \lra b$ in $A$ and each horizontal morphism $g \colon c \lra d$ in $B$, a cell $F(u,g)$ in $C$ as in the left of (\ref{threecells}),
\item for each horizontal morphism $f \colon a \lra b$ in $A$ and each vertical morphism $v \colon c \lra d$ in $B$, a cell $F(f,v)$ in $C$ as in the centre of (\ref{threecells}),
\item for each horizontal morphism $f \colon a \lra b$ in $A$ and each horizontal morphism $g \colon c \lra d$ in $B$, an invertible globular cell $F(f,g)$ in $C$ as in the right of (\ref{threecells}),
\begin{equation} \label{threecells}
\cd[@C=2.5em]{
F(a,c) \ar[r]^-{F(a,g)} \ar[d]_-{F(u,c)} \dtwocell{dr}{F(u,g)} & F(a,d) \ar[d]^-{F(u,d)} \\
F(b,c) \ar[r]_-{F(b,g)} & F(b,d)
}
\quad
\cd[@C=2.5em]{
F(a,c) \ar[r]^-{F(f,c)} \ar[d]_-{F(a,v)} \dtwocell{dr}{F(f,v)} & F(b,c) \ar[d]^-{F(b,v)} \\
F(a,d) \ar[r]_-{F(f,d)} & F(b,d)
}
\quad
\cd[@C=2.5em]{
F(a,c) \ar[r]^-{F(a,g)} \ar[d]_-{F(f,c)} \dtwocell{dr}{F(f,g)} & F(a,d) \ar[d]^-{F(f,d)} \\
F(b,c) \ar[r]_-{F(b,g)} & F(b,d)
}
\end{equation}
\end{enumerate}
subject to the following axioms:
\begin{enumerate}[resume*]
\item for each vertical morphism $u \colon a \lra b$ in $A$, the above data define a vertical transformation $F(u,-) \colon F(a,-) \lra F(b,-) \colon B \lra C$,
\item for each vertical morphism $v \colon c \lra d$ in $B$, the above data define a vertical transformation $F(-,v) \colon F(-,c) \lra F(-,d) \colon A \lra C$,
\item for each horizontal morphism $f \colon a \lra b$ in $A$, the above data define a pseudo horizontal transformation $F(f,-) \colon F(a,-) \lra F(b,-) \colon B \lra C$,
\item for each horizontal morphism $g \colon c \lra d$ in $B$, the above data define a pseudo horizontal transformation $F(-,g) \colon F(-,c) \lra F(-,d) \colon A \lra C$.
\end{enumerate}
\end{definition}

\begin{remark} \label{twovarrmk}
The $n$-ary morphisms of the multicategory $\sf{PsDbl}$ are ``pseudo double functors of $n$ variables'', which consist of a similar set of data subject to a similar set of axioms, and for  $n \geq 3$ are subject to a further set of ``cubical'' axioms involving triples of morphisms. In particular, a nullary morphism is precisely an object of its codomain, and a unary morphism is a pseudo double functor. Specialising to those pseudo double categories whose underlying categories are discrete, i.e.\ to bicategories, these are precisely the ``strong $n$-homomorphisms'' of \cite[\S 1.3]{MR2844536}, which in turn generalise the ``iso-quasi-functors of $n$-variables'' of \cite[\S4.24]{MR0371990} (later renamed ``cubical functors'' \cite[\S 4.1]{MR1261589}).
\end{remark}

\begin{definition} \label{vertdef}
Let $F,G \colon (A,B) \lra C$ be pseudo double functors of two variables. A \emph{vertical transformation of two variables} $\sigma \colon F \lra G$ consists of the following data:
\begin{enumerate}[(i)]
\item a natural transformation $\sigma \colon F \lra G \colon A_0 \times B_0 \lra C_0$,
\item for each object $a \in A$, a vertical transformation $\sigma(a,-) \colon F(a,-) \lra G(a,-) \colon B \lra C$ agreeing with the natural transformation (i) on underlying categories,
\item for each object $c \in B$, a vertical transformation $\sigma(-,c) \colon F(-,c) \lra G(-,c) \colon A \lra C$ agreeing with the natural transformation (i) on underlying categories,
\end{enumerate}
subject to the following axioms:
\begin{enumerate}[resume*]
\item for each horizontal morphism $f \colon a \lra b$ in $A$, the above data define a modification $\sigma(f,-)$ as in (\ref{twomods}), 
\item for each horizontal morphism $g \colon c \lra d$ in $B$, the above data define a modification $\sigma(-,g)$ as in  (\ref{twomods}).
\end{enumerate}
\begin{equation} \label{twomods}
\cd[@C=2.5em]{
F(a,-) \ar[r]^-{F(f,-)} \ar[d]_-{\sigma(a,-)} \dtwocell{dr}{\sigma(f,-)} & F(b,-) \ar[d]^-{\sigma(b,-)} \\
G(a,-) \ar[r]_-{G(f,-)} & G(b,-)
}
\qquad
\cd[@C=2.5em]{
F(-,c) \ar[r]^-{F(-,g)} \ar[d]_-{\sigma(-,c)} \dtwocell{dr}{\sigma(-,g)} & F(-,d) \ar[d]^-{\sigma(-,d)} \\
G(-,c) \ar[r]_-{G(-,g)} & G(-,d)
}
\end{equation}
\end{definition}

\begin{remark}
One can define similarly a notion of ``vertical transformation of $n$ variables''; unlike in the previous remark, no further set of axioms is required for the general case. Note that a vertical transformation of zero variables is simply a vertical morphism. Specialising to bicategories  yields a notion of \emph{multivariable icon}.
\end{remark}

\begin{definition} \label{hordef}
Let $F,G \colon (A,B) \lra C$ be pseudo double functors of two variables. A \emph{pseudo horizontal transformation of two variables} $\theta \colon F \lra G$ consists of the following data:
\begin{enumerate}[(i)]
\item for each pair of objects $(a,c) \in \mathrm{ob} A \times \mathrm{ob} B$, a horizontal morphism $\theta(a,c) \colon F(a,c) \lra G(a,c)$ in $C$,
\item for each object $a \in A$, a pseudo horizontal transformation $\theta(a,-)\colon F(a,-) \lra G(a,-)$ agreeing with (i) on objects,
\item for each object $c \in B$, a pseudo horizontal transformation $\theta(-,c)\colon F(-,c) \lra G(-,c)$ agreeing with (i) on objects,
\end{enumerate}
subject to the following axioms:
\begin{enumerate}[resume*]
\item for each vertical morphism $u \colon a \lra b$ in $A$, the above data define a modification $\theta(u,-)$ as in (\ref{twomoremods}),
\item for each vertical morphism $v \colon c \lra d$ in $B$, the above data define a modification $\theta(-,v)$ as in (\ref{twomoremods}),
\end{enumerate}
\begin{equation} \label{twomoremods}
\cd[@C=2.5em]{
 F(a,-) \ar[r]^-{\theta(a,-)} \ar[d]_-{F(u,-)} \dtwocell{dr}{\theta(u,-)} & G(a,-) \ar[d]^-{G(u,-)} \\
F(b,-) \ar[r]_-{\theta(b,-)} & G(b,-)
}
\qquad
\cd[@C=2.5em]{
F(-,c) \ar[r]^-{\theta(-,c)} \ar[d]_-{F(-,v)} \dtwocell{dr}{\theta(-,v)} & G(-,c) \ar[d]^-{G(-,v)} \\
F(-,d) \ar[r]_-{\theta(-,d)} & G(-,d)
}
\end{equation}
\begin{enumerate}[resume*]
\item for each horizontal morphism $f \colon a \lra b$ in $A$, the above data define an invertible modification $\theta(f,-)$ as in (\ref{yettwomoremods}),
\item for each horizontal morphism $g \colon c \lra d$ in $B$, the above data define an invertible modification $\theta(-,g)$ as in (\ref{yettwomoremods}).
\end{enumerate}
\begin{equation} \label{yettwomoremods}
\cd[@C=2.5em]{
F(a,-) \ar[r]^-{F(f,-)} \ar[d]_-{\theta(a,-)} \dtwocell{dr}{\theta(f,-)} & F(b,-) \ar[d]^-{\theta(b,-)} \\
G(a,-) \ar[r]_-{G(f,-)} & G(b,-)
}
\qquad
\cd[@C=2.5em]{
F(-,c) \ar[r]^-{F(-,g)} \ar[d]_-{\theta(-,c)} \dtwocell{dr}{\theta(-,g)} & F(-,d) \ar[d]^-{\theta(-,d)} \\
G(-,c) \ar[r]_-{G(-,g)} & G(-,d)
}
\end{equation}
\end{definition}

\begin{remark}
One can define similarly a notion of ``pseudo horizontal transformation of $n$ variables''; as in the previous remark, no further set of axioms is required for the general case. Specialising to bicategories yields a notion of ``pseudonatural transformation of $n$ variables'', which in turn generalises the ``iso-quasi-natural transformations'' of $n$ variables of \cite[\S4.24]{MR0371990}.
\end{remark}

\begin{definition} \label{moddef}
Let $\sigma \colon F \lra H$ and $\tau \colon G \lra K$ be vertical transformations of two variables, and let $\theta \colon F \lra G$ and $\varphi \colon H \lra K$ be pseudo horizontal transformations of two variables. A \emph{modification of two variables} $m$ as in (\ref{modtwovarcomp}) consists of a cell $m(a,b)$ in $C$ for each pair of objects $(a,b) \in \mathrm{ob} A \times \mathrm{ob} B$ as in (\ref{modtwovarcomp}), such that $m(a,b)$ defines a modification in each variable.
\begin{equation} \label{modtwovarcomp}
\cd{
F \ar[r]^-{\theta} \ar[d]_-{\sigma} \dtwocell{dr}{m} & G \ar[d]^-{\tau} \\
H \ar[r]_-{\varphi} & K
}
\qquad 
\qquad
\cd[@C=2.5em]{
F(a,b) \ar[r]^-{\theta(a,b)} \ar[d]_-{\sigma(a,b)} \dtwocell{dr}{m(a,b)} & G(a,b) \ar[d]^-{\tau(a,b)} \\
H(a,b) \ar[r]_-{\varphi(a,b)} & K(a,b)
}
\end{equation}
\end{definition}

It is a simple matter of unwinding the definitions to verify that the collection of pseudo double functors of two variables $(A,B) \lra C$, together with the two-variable vertical transformations, pseudo horizontal transformations, and modifications between them, forms a pseudo double category isomorphic to $\underline{\Hom}(A,\underline{\Hom}(B,C))$.

\begin{proposition} \label{skewprop}
There exists a symmetric skew closed structure on the category $\psdbl_{\mathrm{s}}$ of pseudo double categories and strict double functors whose unit object $I$ represents the functor $\mathrm{ob} \colon \psdbl_{\mathrm{s}} \lra \Set$ and whose internal hom objects are the pseudo double categories $\underline{\Hom}(A,B)$.
\end{proposition}
\begin{proof}
A substantial amount of the proof is contained in \cite[\S7]{MR1716779} and \cite[\S2]{MR2303000}. We describe the symmetric skew closed structure; the remaining details are then straightforward.

The unit object $I$ is defined to be the free pseudo double category on the singleton set $1$, and so represents the functor $\mathrm{ob} \colon \psdbl_{\mathrm{s}} \lra \Set$ by definition; as its underlying category is discrete (indeed, it is the terminal category), it is also the free bicategory on $1$, which is described concretely in \cite[\S4.2]{Bourke2015}. The strict double functor $i \colon \underline{\Hom}(I,A) \lra A$ is defined by evaluation at the unique object of $I$. The strict double functor $j \colon I \lra \underline{\Hom}(A,A)$ picks out the identity pseudo double functor on $A$.

The strict double functor $L \colon \underline{\Hom}(B,C) \lra \underline{\Hom}(\underline{\Hom}(A,B),\underline{\Hom}(A,C))$ corresponds to the pseudo double functor of two variables $(\underline{\Hom}(B,C),\underline{\Hom}(A,B)) \lra \underline{\Hom}(A,C)$ defined by horizontal composition, which is strict in the first variable (i.e.\ the pseudo double functors in Definition \ref{deffuntwovar} (iii) are strict).

The symmetry isomorphism $s \colon \underline{\Hom}(A,\underline{\Hom}(B,C)) \lra \underline{\Hom}(B,\underline{\Hom}(A,C))$ sends a pseudo double functor of two variables $F \colon (A,B) \lra C$ to the pseudo double functor of two variables $sF \colon (B,A) \lra C$ whose underlying functor is the composite of the underlying functor of $F$ with the symmetry isomorphism $B_0 \times A_0 \cong A_0 \times B_0$, and whose remaining data (in the terms of Definition \ref{deffuntwovar}) are defined from those of $F$ by the interchanges  (ii) $\leftrightarrow$ (iii), (iv) $\leftrightarrow$ (v),  and where for (vi) we replace $F(f,g)$ with its inverse. The definition of $s$ on vertical transformations, pseudo horizontal transformations, and modifications is given by similar interchanges of the data of Definitions \ref{vertdef}, \ref{hordef}, and \ref{moddef}.
\end{proof}

We define the \emph{symmetric closed multicategory $\sf{PsDbl}$ of pseudo double categories} to be the symmetric closed multicategory of weak morphisms of the symmetric skew closed category $\psdbl_{\mathrm{s}}$ of Proposition \ref{skewprop}. Its multimorphisms are the pseudo double functors of $n$ variables described in Remark \ref{twovarrmk}, and its internal hom objects are the hom pseudo double categories $\underline{\Hom}(A,B)$.
We recover Verity's symmetric multicategory of bicategories as a full sub-multicategory of $\sf{PsDbl}$. Moreover, there is an adjunction of symmetric multicategories 
\begin{equation} \label{appadj}
\cd{
{\sf Bicat} \ar@<1.5ex>[rr]_-{\hdash}^-{} && \ar@<1.5ex>[ll]^-{\mathbf{H}} {\sf PsDbl}
}
\end{equation}
whose left adjoint is the full inclusion and whose right adjoint sends a pseudo double category $A$ to the bicategory $\mathbf{H}A$ defined by discarding the vertical morphisms and non-globular cells of $A$.  (Note that this adjunction arises from an adjunction of symmetric skew closed categories.)

Furthermore, we define the \emph{symmetric multicategory $\sf{Dbl}$ of double categories} to be the (non-full) sub symmetric multicategory of $\sf{PsDbl}$ consisting of the (strict) double categories and the (strict) double functors of $n$ variables, i.e.\ the pseudo double functors of $n$ variables that are strict in each variable. This symmetric multicategory is closed, with internal homs $\underline{\Ps}(A,B)$ given by the full sub double categories of $\underline{\Hom}(A,B)$ on the double functors, and moreover is represented by a symmetric closed monoidal structure on the category $\mathbf{Dbl}$ of double categories and double functors, which generalises Gray's symmetric closed monoidal structure on $\twocat$. The adjunction (\ref{appadj}) restricts to an adjunction between the symmetric multicategory  $\sf{Gray}$ of $2$-categories  represented by the symmetric Gray monoidal structure and the symmetric multicategory $\sf{Dbl}$, which is equivalently an adjunction of symmetric monoidal categories.

To conclude, we address the question of representability of the symmetric multicategories of bicategories and pseudo double categories discussed in \S \ref{intro}. Neither multicategory is representable (as can be seen by the arguments of \cite[\S 1.3]{MR2844536}), however they both admit enrichments to symmetric $2$-multicategories that are \emph{birepresentable}, that is, representable up to equivalence.

We define the symmetric closed $2$-multicategory ${\sf PsDbl}_2$ of pseudo double categories to be the change of base of the canonical self-enrichment of ${\sf PsDbl}$ along the symmetric multifunctor $\mathbf{U} \colon {\sf PsDbl} \lra {\sf Cat}$ that sends a pseudo double category $A$ to its underlying category $\mathbf{U}A = A_0$. (This symmetric multifunctor arises from a symmetric closed functor $\psdbl_{\mathrm{s}} \lra \Cat$.) The full sub-$2$-multicategory of ${\sf PsDbl}_2$ on the bicategories is a symmetric $2$-multicategory ${\sf Bicat}_2$ of bicategories, whose $2$-cells are multivariable icons; note that this $2$-multicategory is not closed, even in the bicategorical up-to-equivalence sense. Let $\Hom_2(A_1,\ldots,A_n;B)$ denote the hom categories of the $2$-multicategory ${\sf PsDbl}_2$.

Note that change of base along the symmetric multifunctor $\mathbf{U} \colon {\sf PsDbl} \lra {\sf Cat}$ defines a  $2$-functor $\mathbf{U}_{\ast} \colon {\sf PsDbl}\text{-}\Cat \lra \twocat$ that sends a ${\sf PsDbl}$-category to its underlying $2$-category. We say that a morphism in a ${\sf PsDbl}$-category is an equivalence if it is an equivalence in the underlying $2$-category. In particular, applied to the ${\sf PsDbl}$-category of pseudo double categories, this is the usual notion of equivalence of pseudo double categories.

\begin{theorem} \label{birep}
For each integer $n \geq 0$ and pseudo double categories $A_1,\ldots,A_n$, there exists a pseudo double functor of $n$ variables $K \colon (A_1,\ldots,A_n) \lra A_1\times \cdots \times A_n$ such that the induced pseudo double functor 
\begin{equation} \label{unidblfun}
\underline{\Hom}(K;1) \colon \underline{\Hom}(A_1 \times \cdots \times A_n,B) \lra \underline{\Hom}(A_1,\ldots,A_n;B)
\end{equation}
is an equivalence of pseudo double categories for each pseudo double category $B$.
\end{theorem}
\begin{proof}
We define the multivariable pseudo double functors $K$ recursively as follows. For $n=0$, define $K \colon () \lra 1$ to be the nullary morphism corresponding to the unique object of the terminal pseudo double category $1$. For $n = 1$, define $K \colon A \lra A$ to be the identity pseudo double functor. For $n = 2$, let $K \colon (A,B) \lra A \times B$ be the pseudo double functor of two variables defined by the following data (numbered as in Definition \ref{deffuntwovar}): 
(i) is the identity on underlying categories, (ii) $K(a,-)$ and (iii) $K(-,b)$ are the composites
\begin{equation*}
\cd[@C=3em]{
B \cong 1 \times B \ar[r]^-{a \times 1_B} & A \times B
}
\qquad
\qquad
\cd[@C=3em]{
A \cong A \times 1 \ar[r]^-{1_A \times b} & A \times B
}
\end{equation*}
respectively, and (iv) $K(u,g)$, (v) $K(f,v)$, and (vi) $K(f,g)$ are the cells in $A \times B$
\begin{equation*}
\cd[@C=1em]{
(a,c) \ar[rr]^-{(1_a^h,g)} \ar[d]_-{(u,1_c^v)} \dtwocell{dr}{(1_u,1_g)} && (a,d) \ar[d]^-{(u,1_d^v)} \\
(b,c) \ar[rr]_-{(1_b^h,g)} && (b,d)
}
\qquad
\cd[@C=1em]{
(a,c) \ar[rr]^-{(f,1_c^h)} \ar[d]_-{(1_a^v,v)} \dtwocell{dr}{(1_f,1_v)} && (b,c) \ar[d]^-{(1_b^v,v)} \\
(a,d) \ar[rr]_-{(f,1_d^h)} && (b,d)
}
\qquad
\cd[@C=1em]{
(a,c) \ar[rr]^-{(1_a^h,g)} \ar[d]_-{(f,1_c^h)} \dtwocell{dr}{(1_f,1_g)} && (a,d) \ar[d]^-{(f,1_d^h)} \\
(b,c) \ar[rr]_-{(1_b^h,g)} && (b,d)
}
\end{equation*}
respectively.
For $n \geq 3$, define $K$ to be the following composite.
\begin{equation*}
\cd{
(A_1,\ldots,A_{n-1},A_n) \ar[r]^-{(K,1)} & (A_1 \times \cdots \times A_{n-1},A_n) \ar[r]^-K & (A_1 \times \cdots \times A_{n-1}) \times A_n \cong A_1 \times \cdots \times A_n
}
\end{equation*}

We first show that the functor 
\begin{equation} \label{ev2}
\cd{
\Hom_2(K;1) \colon \Hom_2(A \times B,C) \lra \Hom_2(A,B;C)
}
\end{equation}
is an equivalence of categories. To prove that it is essentially surjective on objects, let  $F \colon (A,B) \lra C$ be a pseudo double functor of two variables. Define $\overline{F} \colon A \times B \lra C$ to be the pseudo double functor that agrees with $F$ on underlying categories, sends a horizontal morphism $(f,g) \colon (a,c) \lra (b,d)$ in $A \times B$ to the composite
\begin{equation*}
\cd[@C=3em]{
F(a,c) \ar[r]^-{F(f,c)} & F(b,c) \ar[r]^-{F(b,g)} & F(b,d),
}
\end{equation*}
has horizontal unit constraint at the object $(a,c) \in A \times B$ given by the composite of globular cells in $C$
\begin{equation*}
\cd[@C=3em]{
1_{F(a,c)} \ar[r]^-{\cong} & 1_{F(a,c)} \cdot 1_{F(a,c)} \ar[r]^-{\varphi_0\cdot\varphi_0} & F(a,1_c) \cdot F(1_a,c),
}
\end{equation*}
has horizontal composition constraint at the composable pair $(f,g) \colon (a,c) \lra (b,d)$, \linebreak $(h,k) \colon (b,d) \lra (x,y)$ of horizontal morphisms in $A \times B$ given by the following pasting composite of invertible globular cells in $C$,
\begin{equation*}
\cd{
{} & {} & F(b,d) \ar[dr]^-{F(h,d)} \dtwocell{dd}{F(h,g)} & {} & {} \\
{} & F(b,c) \ar[ur]^-{F(b,g)} \dtwocell[0.6]{d}{\varphi} \ar[dr]|-{F(h,c)} & {} & F(x,d) \dtwocell[0.6]{d}{\varphi} \ar[dr]^-{F(x,k)} & {} \\
F(a,c) \ar[ur]^-{F(f,c)} \ar[rr]_-{F(hf,c)} & {} & F(x,c) \ar[rr]_-{F(x,kg)}  \ar[ur]|-{F(x,g)} & {} & F(x,y)
}
\end{equation*}
and sends a cell $(\alpha,\beta)$ in $A \times B$ as below 
\begin{equation*}
\cd[@C=1em]{
(a,c) \ar[d]_-{(u,v)} \ar[rr]^-{(f,g)} &\dtwocell{d}{(\alpha,\beta)} & (b,d) \ar[d]^-{(s,t)} \\
(a',c') \ar[rr]_-{(f',g')} & {}& (b',d')
}
\end{equation*}
to the following pasting composite of cells in $C$.
\begin{equation*}
\cd[@C=2em]{
F(a,c) \ar[rr]^-{F(f,c)} \ar[d]_-{F(u,c)} \dtwocell{drr}{F(\alpha,c)} && F(b,c) \ar[d]^-{F(s,c)} \ar[rr]^-{F(b,g)} & \dtwocell{d}{F(s,g)} & F(b,d) \ar[d]^-{F(s,d)} \\
F(a',c) \ar[d]_-{F(a',v)}  \ar[rr]|-{F(f',c)} &{} \dtwocell{d}{F(f',v)} & F(b',c) \ar[d]^-{F(b',v)} \ar[rr]|-{F(b',g)} & {} \dtwocell{d}{F(b',\beta)} & F(b',d) \ar[d]^-{F(b',t)} \\
F(a',c') \ar[rr]_-{F(f',c')} & {} & F(b',c') \ar[rr]_-{F(b',g')} & {} & F(b',d')
}
\end{equation*}
There exists an invertible vertical transformation of two variables $\sigma \colon F \lra \overline{F}K$ whose underlying natural transformation is the identity, and whose cell components $\sigma(a,g)$ and $\sigma(f,c)$ are given by the following composites.
\begin{equation*}
\cd{
F(a,g)\ar[r]^-{\cong} & F(a,g) \cdot 1_{F(a,c)} \ar[r]^-{1 \cdot \varphi_0} & F(a,g) \cdot F(1_a,c)
}
\end{equation*}
\begin{equation*}
\cd{
F(f,c) \ar[r]^-{\cong} & 1_{F(b,c)} \cdot F(f,c) \ar[r]^-{\varphi_0\cdot 1} & F(b,1_c)\cdot F(f,c)
}
\end{equation*}
Hence the functor (\ref{ev2}) is essentially surjective on objects.

To prove that the functor (\ref{ev2}) is fully faithful, let $F,G \colon A \times B \lra C$ be pseudo double functors, and let $\sigma\colon FK \lra GK$ be a vertical transformation of two variables. Define \linebreak $\overline{\sigma} \colon F \lra G$ to be the vertical transformation that agrees with $\sigma$ on underlying categories, and whose cell component at a horizontal morphism $(f,g) \colon (a,c) \lra (b,d)$ in $A \times B$ is the following composite cell in $C$,
\begin{equation*}
\cd{
F(a,c) \ar[rrrr]^-{F(f,g)} \ar@{=}[d] & {} & {} \ar@{}[d]|-{\vcong} & {} & F(b,d) \ar@{=}[d] \\
F(a,c) \ar[d]_-{\sigma(a,c)} \ar[rr]^-{F(f,1_c)} & {} \dtwocell{d}{\sigma(f,c)} & F(b,c) \ar[d]^-{\sigma(b,c)} \ar[rr]^-{F(1_b,g)} & {} \dtwocell{d}{\sigma(b,g)} & F(b,d) \ar[d]^-{\sigma(b,d)} \\
G(a,c) \ar@{=}[d] \ar[rr]_-{G(f,1_c)} & {} & G(b,c) \ar[rr]_-{G(1_b,g)} \ar@{}[d]|-{\vcong} & {} & G(b,d) \ar@{=}[d]\\
G(a,c) \ar[rrrr]_-{G(f,g)} & {} & {} & {} & G(b,d)
}
\end{equation*}
where the unlabelled isomorphisms are given by the pseudo double functor constraints of $F$ and $G$. This is the unique vertical transformation $\overline{\sigma} \colon FK \lra GK$ such that $\overline{\sigma}K = \sigma$. Hence the functor (\ref{ev2}) is fully faithful, and is therefore an equivalence of categories.

We now prove by induction on $n \geq 0$ that the pseudo double functors (\ref{unidblfun}) are equivalences of pseudo double categories. For $n=0$, and for any pseudo double category $X$, the functor 
\begin{equation*}
\cd{
\Hom_2(X,\underline{\Hom}(K;1)) \colon \Hom_2(X,\underline{\Hom}(1,B)) \lra \Hom_2(X,B)
}
\end{equation*}
is equal to the composite of equivalences 
\begin{align*}
\Hom_2(X,\underline{\Hom}(1,B)) &\cong \Hom_2(X,1;B) \\
&\simeq \Hom_2(X \times 1,B) \\
&\cong \Hom_2(X,B)
\end{align*}
and is therefore an equivalence of categories. Hence, by the bicategorical Yoneda lemma \cite[\S 1.9]{MR574662},  the pseudo double functor $\underline{\Hom}(K;1) \colon \underline{\Hom}(1,B) \lra B$ is an equivalence of pseudo double categories. For $n=1$, the pseudo double functor (\ref{unidblfun}) is an identity and is therefore an equivalence. For $n=2$, and for any pseudo double category $X$, the functor
\begin{equation*}
\cd{
\Hom_2(X,\underline{\Hom}(K;1)) \colon \Hom_2(X,\underline{\Hom}(A\times B,C)) \lra \Hom_2(X,\underline{\Hom}(A,B;C))
}
\end{equation*}
is naturally isomorphic to the composite of equivalences
\begin{align*}
\Hom_2(X,\underline{\Hom}(A\times B,C)) &\cong \Hom_2(X,A \times B;C) \\
&\simeq \Hom_2(X \times (A \times B);C) \\
&\cong \Hom_2((X \times A) \times B;C) \\
&\cong \Hom_2(X \times A,\underline{\Hom}(B,C)) \\
&\simeq \Hom_2(X,A;\underline{\Hom}(B,C)) \\
&\cong \Hom_2(X,\underline{\Hom}(A,B;C))
\end{align*}
and is therefore an equivalence of categories. Hence, by the bicategorical Yoneda lemma,  the pseudo double functor $\underline{\Hom}(K;1) \colon \underline{\Hom}(A\times B,C) \lra \underline{\Hom}(A,B;C)$ is an equivalence of pseudo double categories.

For $n \geq 3$, we have by induction that the pseudo double functor (\ref{unidblfun}) is equal to the composite of equivalences
\begin{align*}
\underline{\Hom}(A_1 \times \cdots \times A_n,B) &\cong \underline{\Hom}((A_1 \times \cdots \times A_{n-1}) \times A_n,B) \\
&\simeq \underline{\Hom}(A_1 \times \cdots \times A_{n-1},A_n;B) \\
&\cong \underline{\Hom}(A_1 \times \cdots \times A_{n-1},\underline{\Hom}(A_n,B)) \\
&\simeq \underline{\Hom}(A_1,\ldots,A_{n-1};\underline{\Hom}(A_n,B)) \\
&\cong \underline{\Hom}(A_1,\ldots,A_n;B)
\end{align*}
and is therefore an equivalence of pseudo double categories.
\end{proof}

For each $n \geq 0$ and  bicategories $A_1,\ldots,A_n$, and $B$, Theorem \ref{birep} establishes equivalences of pseudo double categories 
\begin{equation} \label{bicatequiv}
\underline{\Hom}(A_1 \times \cdots \times A_n,B) \simeq \underline{\Hom}(A_1,\ldots,A_n;B),
\end{equation}
and hence in particular equivalences of categories
\begin{equation*}
\Hom_2(A_1 \times \cdots \times A_n,B) \simeq \Hom_2(A_1,\ldots,A_n;B)
\end{equation*}
between categories of (multivariable) pseudofunctors and icons. Since each pseudo double category $\underline{\Hom}(A_1,\ldots,A_n;B)$ is fibrant (in the sense of 
\cite[Definition 18]{MR2496344}), the equivalences of pseudo double categories (\ref{bicatequiv}) induce biequivalences of bicategories
\begin{equation*}
\Hom(A_1 \times \cdots \times A_n,B) \sim \Hom(A_1,\ldots,A_n;B),
\end{equation*}
and in particular we recover the biequivalences 
\begin{equation*}
\Hom(A \times B,C) \sim \Hom(A,\Hom(B,C)) \qquad \Hom(1,A) \sim A
\end{equation*}
of \cite[\S 1.34]{MR574662}.
\bibliographystyle{alpha}

\end{document}